\documentclass[10pt, final]{article}
\usepackage{a4}
\usepackage{amsmath}%
\usepackage{amstext}%
\usepackage{amssymb}%
\usepackage{showkeys}%
\usepackage{epsfig}%
\usepackage{cite}
\usepackage{tikz}
\usepackage{subfig}
\usepackage{caption}
\usepackage{multirow}
\usepackage{booktabs}
\usepackage{floatrow}

\usepackage{listings}
\newtheorem{theorem}{Theorem}

\newtheorem{axiom}{Axiom}

\newtheorem{corollary}[theorem]{Corollary}

\newtheorem{definition}[axiom]{Definition}

\newtheorem{lemma}[theorem]{Lemma}

\newtheorem{rem}[theorem]{Remark}
\newenvironment{remark}{\rem\rm}{\endrem}

\newcounter{unnumber}

\newtheorem{prf}[unnumber]{Proof.}
\newenvironment{proof}{\prf\rm}{\hfill{$\blacksquare$}\endprf}
\newcommand{\R}{\mathbb{R}}%
\newcommand{\N}{\mathbb{N}}%
\newcommand{\ol}{\overline}%
\newcommand{\ul}{\underline}%

\renewcommand{\>}{\right\rangle}
\newcommand{\<}{\left\langle}
\newcommand{\To}{\longrightarrow}

\DeclareMathOperator*\dom{dom}%
\DeclareMathOperator*\gr{Gr}%
\DeclareMathOperator*\id{Id}%
\DeclareMathOperator*\prox{prox}%
\DeclareMathOperator*\argmin{argmin}

\DeclareMathOperator*\zer{zer}

\DeclareMathOperator*\fix{Fix}
\DeclareMathOperator*\crit{crit}
\DeclareMathOperator*\dist{dist}
\DeclareMathOperator*\proj{proj}
\DeclareMathOperator*\oR{\overline{\R}}%

\title{Continuous dynamics related to monotone inclusions and non-smooth optimization problems}

\author{
Ern\"{o} Robert Csetnek \thanks {University of Vienna, Faculty of Mathematics, Oskar-Morgenstern-Platz 1, A-1090 Vienna, Austria,
email: ernoe.robert.csetnek@univie.ac.at. Research supported by FWF (Austrian Science Fund), project P 29809-N32.}}

\begin{document}
\maketitle

\noindent \textbf{Abstract.} The aim of this survey is to present the main important techniques and tools from
variational analysis used for first and second order dynamical systems of implicit type for solving monotone
inclusions and non-smooth optimization problems. The differential equations are expressed by means
of the resolvent (in case of a maximally monotone set valued operator) or the proximal operator
for non-smooth functions. The asymptotic analysis of the trajectories generated relies on Lyapunov theory,
where the appropriate energy functional plays a decisive role. While the most part of the paper is related
to monotone inclusions and convex optimization problems in the variational case, we present
also results for dynamical systems for solving non-convex optimization problems, where the Kurdyka-\L{}ojasiewicz property is used.

\vspace{1ex}

 \vspace{1ex}


\section{Introduction and preliminaries}\label{sec-intr}

Dynamical systems approaching monotone inclusions and optimization problems enjoy much attention since the seventies of the last century (Br\'{e}zis, Baillon and Bruck, see \cite{brezis, baillon-brezis1976, bruck}), not only due to their intrinsic importance in areas like differential equations and applied functional analysis, but also because they have been recognized as a valuable tool for discovering and studying numerical algorithms for optimization problems obtained by time discretization of the continuous dynamics. The dynamic approach to iterative methods in optimization can furnish deep insights into the expected behavior of the method and the techniques used in the continuous case can be adapted to obtain results for the discrete algorithm. For more on the relations between the continuous and discrete dynamics we refer the reader to \cite{peyp-sorin2010}.

Let us mention that the discretization $\dot x(t)\approx \frac{1}{\gamma}(x_{n+1}-x_n$) (with $\gamma >0$) of the gradient flow $$\dot x(t)=-\nabla g(x(t)),$$ where
$g:{\cal H}\to \R$ is a smooth function, leads to the well known steepest descent method (gradient method)
$$x_{n+1}=x_n-\gamma\nabla g(x_n),$$ used for solving the minimization problem
$$\min_{x\in {\cal H}} g(x).$$
In the case of convex non-smooth optimization problems, the discretization of the differential inclusion
$$\dot x(t)\in-\partial f(x(t))$$ leads to the subgradient method $x_{n+1}\in x_n-\gamma\partial f(x_n)$, or, when we use the discretization
$x_{n+1}-x_n\in\gamma\partial f(x_{n+1})$, we obtain the proximal point algorithm \cite{rock-prox}
$$x_{n+1}=(\id+\gamma\partial f)^{-1}(x_n).$$
Let us proceed and consider continuous {\it implicit-type dynamical
systems} associated with monotone inclusions/optimization problems, which are ordinary differential equations formulated via resolvents of maximally monotone operators.
In \cite{bolte-2003}, Bolte
studied the convergence of the trajectories of the following dynamical system
\begin{equation}\label{intr-syst-bolte}
\dot x(t)+x(t)=\proj\nolimits_C\big(x(t)-\gamma\nabla g(x(t))\big)
\end{equation}
where $C$ is a nonempty, closed and convex subset of
${\cal H}$, $x_0\in {\cal H}$, and $\proj_C$ denotes the projection operator on the set $C$. It has been shown
in the convex setting that the
trajectory of \eqref{intr-syst-bolte} converges weakly to a minimizer of the optimization problem
\begin{equation}\label{intr-opt-bolte}
\inf_{x\in C}g(x),
\end{equation}
provided the latter is solvable.  We refer also to the work of Antipin \cite{antipin} for further statements and results concerning
\eqref{intr-syst-bolte}.

The following generalization of the dynamical system \eqref{intr-syst-bolte} has been considered by Abbas and Attouch in \cite[Section 5.2]{abbas-att-arx14}:
\begin{equation}\label{intr-syst-abb-att}
\dot x(t)+x(t)=\prox\nolimits_{\gamma f}\big(x(t)-\gamma B(x(t))\big),\end{equation}
where $B$ is a cocoercive operator. According to \cite{abbas-att-arx14}, in case $\zer(\partial f+B)\neq\emptyset$, the weak asymptotical convergence of the orbit $x$ of
\eqref{intr-syst-abb-att} to an element in $\zer(\partial f+B)\neq\emptyset$ is
ensured by choosing the step-size $\gamma$ in a suitable domain bounded by the parameter of cocoercivity of the operator $B$.

Let us also mention that dynamical systems of implicit type have been considered in the literature also by Attouch and Svaiter in
\cite{att-sv2011}, Attouch, Abbas and Svaiter in \cite{abbas-att-sv} and Attouch, Alvarez and Svaiter in \cite{att-alv-sv}.

For the minimization of the smooth and convex function $g:{\cal H}\rightarrow\R$ over the nonempty, convex and closed set
$C \subseteq {\cal H}$, a continuous in time second order gradient-projection approach
has been considered in \cite{att-alv, antipin}, having as starting point the dynamical system
\begin{equation}\label{intr-gr-pr}
\ddot x(t) + \gamma\dot x(t) + x(t) = \proj\nolimits_C(x(t)-\eta\nabla g (x(t))),
\end{equation}
with constant damping parameter $\gamma > 0$ and constant step size $\eta > 0$. The system \eqref{intr-gr-pr} becomes in case $C={\cal H}$
the "heavy ball method'', sometimes called also ''heavy ball method with friction''. This nonlinear oscillator with damping
is in case ${\cal H}=\R^2$ a simplified version of the differential
system describing the motion of a heavy ball that rolls over the graph of $g$ and that
keep rolling under its own inertia until friction stop it at a critical point of $g$ (see \cite{att-g-r}).
Later, (see \cite{su-boyd-candes, att-c-p-r-math-pr2018} and the references therein) it was highlighted the fact that a vanishing in time damping is more suitable in the context of second order dynamics, since
this will induce fast convergence results in comparison to first order dynamics.

Second order dynamics provide through discretization numerical schemes with inertial terms, where the new iterate 
is computed based on the previous two ones. Taking into account the history of the iterates makes the asymptotic 
analysis more involved and may provide better convergence rates. It is an interesting topic to take into consideration more than two iterates in order to compute the next one. In case of algorithms we refer the reader to
the recent contribution \cite{comb-glod} and to \cite{att-ch-r-third-ord-optimization} for a third order gradient evolution system. In this survey we are considering first and second order dynamics. 

Finally, motivated by concrete applications where non-convex objects appear, let us say a few words about this setting. In the context of minimizing a non-convex smooth function,
several first- and second-order gradient type dynamical systems have been investigated by
\L{}ojasiewicz \cite{lojasiewicz1963}, Simon \cite{simon}, Haraux and Jendoubi \cite{h-j},
Alvarez, Attouch, Bolte and  Redont \cite[Section 4]{alv-att-bolte-red}, Bolte, Daniilidis and Lewis \cite[Section 4]{b-d-l2006}, etc.
In the aforementioned papers, the convergence of the trajectories  is obtained in the framework of functions satisfying the Kurdyka-\L{}ojasiewicz property, see Section 4 for details.

{\bf 1.1 Contents of the paper.} In this manuscript we revisit the most important first and second-order dynamical
systems of implicit type studied in the context of solving monotone inclusions and non-smooth optimization
problems (convex and non-convex). The main focus is on the techniques which are used in the investigation
of the asymptotic behaviour of the trajectories generated. Also the connection to
discrete schemes will be also underlined in order to see the algorithmic impact of the dynamical systems considered.
Moreover, the proofs are not in every stage of this survey complete, but we will always refer to
appropriate references when something is missing. The existence and uniqueness of the
trajectory usually follows from the Cauchy-Lipschitz-Picard Theorem and its variants and will be omitted
in the following. However, the reader will find the corresponding details in the literature mentioned here.
Let us also mention that there are no new results in this survey, apart from Theorem \ref{cont-ista}, which is
a generalization to the continuous setting of the $O\left(\frac{1}{n}\right)$ rate for objective function
values for prox-gradient schemes with no acceleration (see also ISTA in \cite{BecTeb09}).

The main challenge in the analysis of the dynamical systems is defining an appropriate energy functional in
order to conduct the investigations in the framework of Lyapunov theory. We start with
first order dynamics governed by a nonexpansive operator with the aim of approaching the set of its fixed
points \cite{b-c-dyn-KM}. While at a first look, this seems to be quite far from monotone inclusions or convex optimization
problems, let us underline that this offers a setting in which also the aforementioned problems can be
embedded. In this setting the energy functional is usually expressed by means of the distance of the trajectory to a solution of the problem under consideration. The continuous version of the Opial Lemma plays a decisive role on the analysis. As particular cases we mention the dynamical systems
of forward-backward type, forward-backward-forward or Douglas-Rachford. The case of convex optimization problems
is also considered.

Afterwards we continue with second-order dynamical systems \cite{b-c-sicon2016}. These are motivated by the
fact that by time discretization inertial terms are induced in the algorithmic schemes which are
responsible for acceleration (see also the works of Polyak \cite{polyak}, Nesterov \cite{nes, Nes83}, Bertsekas \cite{bertsekas} in the discrete case).
Also in this setting we start with a very general dynamical system formulated by means of a cocoercive operator
and then consider particular cases for monotone inclusions and convex optimization problems. Relations
to other types of dynamics considered in the literature are also mentioned.

Finally, we consider dynamical systems for non-convex and non-smooth optimization problems \cite{bc-forder-kl}.
How to find the appropriate energy functional in this context is not trivial.
The approach for proving asymptotic convergence for the trajectory towards a critical point of the objective
function, uses three main ingredients
(see \cite{alv-att-bolte-red} for the continuous case and also \cite{att-b-sv2013, b-sab-teb} for a similar approach in the discrete setting).
Namely, we show a sufficient decrease property along the trajectories of a regularization of the objective function, the existence of a
subgradient lower bound for the trajectories and,
finally, we obtain convergence by making use of the Kurdyka-\L{}ojasiewicz property of the objective function.

{\bf 1.2 Notations.} Let us fix a few notations used throughout the paper. Let $\N= \{0,1,2,...\}$ be the set of nonnegative integers. Let ${\cal H}$ be a real Hilbert space with inner product
$\langle\cdot,\cdot\rangle$ and associated norm $\|\cdot\|=\sqrt{\langle \cdot,\cdot\rangle}$.

For readers convenience let us recall some standard notions and results in monotone operator theory
which will be used in the following (see also \cite{bo-van, bauschke-book, simons}). For an arbitrary set-valued operator $A:{\cal H}\rightrightarrows {\cal H}$ we denote by
$\gr A=\{(x,u)\in {\cal H}\times {\cal H}:u\in Ax\}$ its graph.
We use also the notation $\zer A=\{x\in{\cal{H}}:0\in Ax\}$ for the set of zeros of $A$. We say that $A$ is monotone, if $\langle x-y,u-v\rangle\geq 0$ for all $(x,u),(y,v)\in\gr A$. A monotone operator $A$ is said to be maximally monotone, if there exists no proper monotone extension of the graph of $A$ on ${\cal H}\times {\cal H}$.
The resolvent of $A$, $J_A:{\cal H} \rightrightarrows {\cal H}$, is defined by $J_A=(\id_{{\cal H}}+A)^{-1}$, where $\id_{{\cal H}} :{\cal H} \rightarrow {\cal H}, \id_{\cal H}(x) = x$ for all $x \in {\cal H}$, is the identity operator on ${\cal H}$. Moreover, if $A$ is maximally monotone, then $J_A:{\cal H} \rightarrow {\cal H}$ is single-valued and maximally monotone
(see \cite[Proposition 23.7 and Corollary 23.10]{bauschke-book}). For an arbitrary $\gamma>0$ we have (see \cite[Proposition 23.2]{bauschke-book})
\begin{equation}p\in J_{\gamma A}x \ \mbox{if and only if} \ (p,\gamma^{-1}(x-p))\in\gr A.\end{equation}

The operator $A$ is said to be uniformly monotone if there exists an increasing function
$\phi_A : [0,+\infty) \rightarrow [0,+\infty]$ that vanishes only at $0$, and
$\langle x-y,u-v \rangle \geq \phi_A \left( \| x-y \|\right)$ for every $(x,u)\in\gr A$ and $(y,v) \in \gr A$.
If this property is fulfilled in the special case $\phi_A(t)=\gamma t^2$, with $\gamma > 0$, , we say that $A$ is $\gamma$-strongly monotone. We consider also the class of cocoercive operators: the single valued mapping
$B:{\cal H}\rightarrow {\cal H}$ is $\gamma$-cocoercive,
if $\langle x-y,Bx-By\rangle\geq \gamma\|Bx-By\|^2$ for all $x,y\in {\cal H}$. Notice that this is nothing else than  $B^{-1}$ is $\gamma$-strongly monotone.

We recall some standard notations and facts in convex analysis. For a proper, convex and
lower semicontinuous function $f:{\cal H}\rightarrow\R\cup\{+\infty\}$, its (convex) subdifferential at $x\in {\cal H}$ is defined as
$$\partial f(x)=\{u\in {\cal H}:f(y)\geq f(x)+\<u,y-x\> \ \forall y\in {\cal H}\}.$$ When seen as a set-valued mapping, it is a
maximally monotone operator (see \cite{rock}) and for $\gamma >0$, its resolvent is given by
$J_{\gamma \partial f}=\prox_{\gamma f}$ (see \cite{bauschke-book}),
where $\prox_{\gamma f}:{\cal H}\rightarrow {\cal H}$,
\begin{equation}\label{prox-def}\prox\nolimits_{\gamma f}(x)=\argmin_{y\in {\cal H}}\left \{f(y)+\frac{1}{2\gamma}\|y-x\|^2\right\},
\end{equation}
denotes the proximal point operator of $f$.
The function $f$ is said to be
$\nu$-strongly convex, where $\nu>0$, if $f-(\nu/2)\|\cdot\|^2$ is a convex function. Let us mention that if $f$ is
$\nu$-strongly convex, then $\partial f$ is $\nu$-strongly monotone, see \cite[Example 22.3(iv)]{bauschke-book}.

\section{First order dynamical systems}

Let $T:{\cal H}\rightarrow {\cal H}$ be a nonexpansive mapping (that is $\|Tx-Ty\|\leq\|x-y\|$ for all $x,y\in{\cal H}$),
$\lambda:[0,+\infty)\rightarrow [0,1]$ be a Lebesgue measurable function and $x_0\in {\cal H}$. We are concerned with the following dynamical system:

\begin{equation}\label{dyn-syst-KM}\left\{
\begin{array}{ll}
\dot x(t)=\lambda(t)\big(T(x(t))-x(t)\big)\\
x(0)=x_0.
\end{array}\right.\end{equation}

As in \cite{att-sv2011, abbas-att-sv}, we consider the following definition of an absolutely continuous function.

\begin{definition}\label{abs-cont} \rm (see, for instance, \cite{att-sv2011, abbas-att-sv}) A function $f:[0,b]\rightarrow {\cal H}$ (where $b>0$) is said to be absolutely continuous if one of the
following equivalent properties holds:

(i)  there exists an integrable function $g:[0,b]\rightarrow {\cal H}$ such that $$f(t)=f(0)+\int_0^t g(s)ds \ \ \forall t\in[0,b];$$

(ii) $f$ is continuous and its distributional derivative is Lebesgue integrable on $[0,b]$;

(iii) for every $\varepsilon > 0$, there exists $\eta >0$ such that for any finite family of intervals $I_k=(a_k,b_k)\subseteq [a,b]$ we have:
$$\left(I_k\cap I_j=\emptyset \mbox{ and }\sum_k|b_k-a_k| < \eta\right)\Longrightarrow \sum_k\|f(b_k)-f(a_k)\| < \varepsilon.$$
\end{definition}

\begin{remark}\label{rem-abs-cont}\rm (a) It follows from the definition that an absolutely continuous function is differentiable almost
everywhere, its derivative coincides with its distributional derivative almost everywhere and one can recover the function from its derivative $f'=g$
by the integration formula (i).

(b) If $f:[0,b]\rightarrow {\cal H}$ (where $b>0$) is absolutely continuous and $B:{\cal H}\rightarrow {\cal H}$ is $L$-Lipschitz continuous
(where $L\geq 0$), then the function $h=B\circ f$ is absolutely continuous. This can be easily verified by considering the characterization in
Definition \ref{abs-cont}(iii). Moreover, $h$ is almost everywhere differentiable and the inequality $\|h'(\cdot)\|\leq L\|f'(\cdot)\|$ holds almost everywhere.
\end{remark}

We say that $f:[0,+\infty]\rightarrow {\cal H}$ is locally absolutely continuous if it is absolutely continuous 
on every compact interval $[0,b]$, with $b > 0$.

\begin{definition}\label{str-sol}\rm We say that $x:[0,+\infty)\rightarrow {\cal H}$ is a strong global solution of \eqref{dyn-syst-KM} if the
following properties are satisfied:

(i) $x:[0,+\infty)\rightarrow {\cal H}$ is locally absolutely continuous;

(ii) $\dot x(t)=\lambda(t)\big(T(x(t))-x(t)\big)$ for all $t\in[0,+\infty)$;

(iii) $x(0)=x_0$.
\end{definition}

The existence and uniqueness of strong global solutions of \eqref{dyn-syst-KM} is a consequence of the Cauchy-Lipschitz theorem for absolutely continues trajectories (see for example
\cite[Proposition 6.2.1]{haraux}, \cite[Theorem 54]{sontag}).

In order to proceed with the asymptotic analysis, we need the following preparatory results.
\begin{lemma}\label{fejer-cont1} (\!\!\cite[Lemma 5.1]{abbas-att-sv}) Suppose that $F:[0,+\infty)\rightarrow\R$ is locally absolutely continuous and bounded from  below and that
there exists $G\in L^1([0,+\infty))$ such that for almost every $t \in [0,+\infty)$ $$\frac{d}{dt}F(t)\leq G(t).$$
Then there exists $\lim_{t\rightarrow \infty} F(t)\in\R$.
\end{lemma}

\begin{lemma}\label{fejer-cont2} (\!\!\cite[Lemma 5.2]{abbas-att-sv}) If $1 \leq p < \infty$, $1 \leq r \leq \infty$, $F:[0,+\infty)\rightarrow[0,+\infty)$ is
locally absolutely continuous, $F\in L^p([0,+\infty))$, $G:[0,+\infty)\rightarrow\R$, $G\in  L^r([0,+\infty))$ and
for almost all $t$ $$\frac{d}{dt}F(t)\leq G(t),$$ then $\lim_{t\rightarrow +\infty} F(t)=0$.
\end{lemma}
The next result which we recall here is the continuous version of the Opial Lemma (see for example \cite[Lemma 5.3]{abbas-att-sv}, \cite[Lemma 1.10]{abbas-att-arx14}).

\begin{lemma}\label{opial} Let $S \subseteq {\cal H}$ be a nonempty set and $x:[0,+\infty)\rightarrow{\cal H}$ a given map. Assume that

(i) for every $z\in S$, $\lim_{t\rightarrow+\infty}\|x(t)-z\|$ exists;

(ii) every weak sequential cluster point of the map $x$ belongs to $S$.

\noindent Then there exists $x_{\infty}\in S$ such that $w-\lim_{t\rightarrow+\infty}x(t)=x_{\infty}$.
\end{lemma}

\begin{theorem}\label{conv-KM} Let $T:{\cal H}\rightarrow {\cal H}$ be a nonexpansive mapping such that $\fix T\neq\emptyset$,
$\lambda:[0,+\infty)\rightarrow [0,1]$ a Lebesgue measurable function and
$x_0\in {\cal H}$. Suppose that either
$$\int_0^{+\infty}\lambda(t)(1-\lambda(t))dt=+\infty \ \mbox{or} \ \inf_{t\geq 0}\lambda(t)>0.$$
Let  $x:[0,+\infty)\rightarrow{\cal H}$ be the unique strong global solution of \eqref{dyn-syst-KM}. Then the following statements are true:

(i) the trajectory $x$ is bounded and $\int_0^{+\infty}\|\dot x(t)\|^2dt<+\infty$;

(ii) $\lim_{t\rightarrow+\infty}(T(x(t))-x(t))=0$;

(iii) $\lim_{t\rightarrow+\infty}\dot x(t)=0$;

(iv) $x(t)$ converges weakly to a point in $\fix T$, as $t\rightarrow+\infty$.
\end{theorem}

\begin{proof} We rely on Lyapunov analysis combined with the Opial Lemma. We take an arbitrary $y\in\fix T$.
From the nonexpansiveness of $T$ and \cite[Corollary 2.14]{bauschke-book} we obtain:
\begin{align*}
\frac{d}{dt}\|x(t)-y\|^2   = & \ 2\<\dot x(t),x(t)-y\>=\|\dot x(t)+x(t)-y\|^2-\|x(t)-y\|^2-\|\dot x(t)\|^2\\
                           = & \ \|\lambda(t)(T(x(t))-y)+(1-\lambda(t))(x(t))-y)\|^2-\|x(t)-y\|^2-\|\dot x(t)\|^2\\
                          = & \ \lambda(t)\|T(x(t))-y\|^2+(1-\lambda(t))\|x(t)-y\|^2 \\
                            & - \lambda(t)(1-\lambda(t))\|T(x(t))-x(t)\|^2 -\|x(t)-y\|^2-\|\dot x(t)\|^2\\
                         \leq & \ -\lambda(t)(1-\lambda(t))\|T(x(t))-x(t)\|^2-\|\dot x(t)\|^2.
\end{align*}

Hence for all $t\geq 0$ we have that
\begin{equation}\label{ineq-x-dtx}\frac{d}{dt}\|x(t)-y\|^2+\lambda(t)(1-\lambda(t))\|T(x(t))-x(t)\|^2+\|\dot x(t)\|^2\leq 0.\end{equation}

Since $\lambda(t)\in[0,1]$ for all $t\geq 0$, from \eqref{ineq-x-dtx} it follows that $t\mapsto\|x(t)-y\|$ is decreasing,
hence $\lim_{t\rightarrow+\infty}\|x(t)-y\|$ exists.
From here we obtain the boundedness of the trajectory and by integrating \eqref{ineq-x-dtx} we deduce also that $\int_0^{+\infty}\|\dot x(t)\|^2dt<+\infty$ and
\begin{equation}\label{int-l-T}\int_0^{+\infty}\lambda(t)(1-\lambda(t))\|T(x(t))-x(t)\|^2dt<+\infty,\end{equation}
thus (i) holds. Since $y\in\fix T$ has been chosen arbitrary, the first assumption in the continuous version of Opial Lemma is fulfilled.

We show in the following that $\lim_{t\rightarrow+\infty}\|T(x(t))-x(t)\|$ exists and it is a real number. This is immediate if we show that
the function $t\mapsto\frac{1}{2}\|T(x(t))-x(t)\|^2$ is decreasing. According to Remark \ref{rem-abs-cont}(b), the function
$t\mapsto T(x(t))$ is almost everywhere differentiable and $\|\frac{d}{dt}T(x(t))\|\leq \|\dot x(t)\|$ holds for almost all $t\geq 0$.
Moreover, by the first equation of \eqref{dyn-syst-KM} we have
\begin{align*}
\frac{d}{dt}\left(\frac{1}{2}\|T(x(t))-x(t)\|^2\right) = & \<\frac{d}{dt}T(x(t))-\dot x(t),T(x(t))-x(t)\>\\
                                                       = & -\<\dot x(t),T(x(t))-x(t)\>+\<\frac{d}{dt}T(x(t)),T(x(t))-x(t)\>\\
                                                       = & -\lambda(t)\|T(x(t))-x(t)\|^2+\<\frac{d}{dt}T(x(t)),T(x(t))-x(t)\>\\
                                                     \leq & -\lambda(t)\|T(x(t))-x(t)\|^2+\|\dot x(t)\| \cdot \|T(x(t))-x(t)\|=0,
\end{align*}
hence $\lim_{t\rightarrow + \infty}\|T(x(t))-x(t)\|$ exists and is a real number.

(a) Firstly, let us assume that $\int_0^{+\infty}\lambda(t)(1-\lambda(t))dt=+\infty$. This immediately implies by \eqref{int-l-T}
that $\lim_{t\rightarrow+\infty}(T(x(t))-x(t))=0$, thus (ii) holds. Taking into account that $\lambda$ is bounded,
from \eqref{dyn-syst-KM} and (ii) we deduce (iii). For the last property of the theorem we need to verify the second assumption
of the Opial Lemma. Let $\ol x\in {\cal H}$ be a weak sequential cluster point of $x$, that is, there exists a sequence $t_n\rightarrow+\infty$
(as $n\rightarrow\infty$) such that $(x(t_n))_{n\in\N}$ converges weakly to $\ol x$. Applying the demiclosedness principle
\cite[Corollary 4.18]{bauschke-book} and (ii) we obtain $\ol x\in\fix T$
and the conclusion follows.

(b) We suppose now that $\inf_{t\geq 0}\lambda(t)>0$. From the first relation of \eqref{dyn-syst-KM} and (i) we easily deduce
that $Tx-x\in L^2([0,\infty),{\cal H})$, hence the function $t\mapsto\frac{1}{2}\|T(x(t))-x(t)\|^2$ belongs to $L^1([0,\infty))$. Since
$\frac{d}{dt}\left(\frac{1}{2}\|T(x(t))-x(t)\|^2\right)\leq 0$ for almost all $t\geq 0$, we obtain by applying Lemma \ref{fejer-cont2} that
$\lim_{t\rightarrow\infty}\|T(x(t))-x(t)\|^2=0$, thus (ii) holds. The rest of the proof can be done in the lines of case (a) considered above.
\end{proof}

\begin{remark} We refer the reader to \cite{c-eb-tam} for a recent contribution concerning convergence rates of the trajectories generated by \eqref{dyn-syst-KM} in case the operator $T$ satisfies some regularity assumptions, like 
boundedly linear/H\"{o}lder regularity. 
\end{remark}

\begin{remark} The explicit discretization of \eqref{dyn-syst-KM} with respect to the time variable $t$, with step size $h_n>0$,
yields for an initial point $x_0$ the following iterative scheme:
$$x_{n+1}=x_n+h_n\lambda_n(Tx_n-x_n) \ \forall n \geq 0.$$
By taking $h_n=1$ this becomes \begin{equation}\label{KM-discrete}x_{n+1}=x_n+\lambda_n(Tx_n-x_n) \ \forall n \geq 0,\end{equation}
which is the classical Krasnosel'ski\u{\i}--Mann algorithm for finding the set of fixed points of the nonexpansive operator $T$
(see \cite[Theorem 5.14]{bauschke-book}). Let us mention that the convergence of \eqref{KM-discrete} is guaranteed under the condition
$\sum_{n \in \N} \lambda_n(1-\lambda_n)=+\infty$. Notice that in case $\lambda_n=1$ for all $n \in\N$ and for an initial point $x_0$ different from $0$,
the convergence of \eqref{KM-discrete} can fail, as it happens for instance for the operator $T=-\id$.
In contrast to this, as pointed out in Theorem \ref{conv-KM}, the dynamical system \eqref{dyn-syst-KM} has a strong global solution and the convergence
of the trajectory is guaranteed also in case $\lambda(t)=1$ for all $t\geq 0$.
\end{remark}

\begin{remark}\label{average}
The conclusions of Theorem \ref{conv-KM} remain valid for the class of averaged operators.
Let $\alpha\in(0,1)$ be fixed. We say that $R:{\cal H}\rightarrow{\cal H}$ is {\it $\alpha$-averaged} if there exists a nonexpansive operator
$T:{\cal H}\rightarrow{\cal H}$ such that $R=(1-\alpha)\id+\alpha T$. In this case we consider  $\lambda:[0,+\infty)\rightarrow [0,1/\alpha]$ and the condition
$\int_0^{+\infty}\lambda(t)(1-\lambda(t))dt=+\infty$ is replaced by $\int_0^{+\infty}\lambda(t)(1-\alpha\lambda(t))dt=+\infty$.
\end{remark}

\begin{remark} Relying on time rescaling arguments and using the link with a dynamical system governed by a
cocoercive operator, it has been shown in \cite[Section 4]{b-c-dyn-KM} that the conclusion of the
theorem remains valid under the condition $\int_0^{+\infty}\lambda(t)dt=+\infty$, which is weaker than
the ones considered above. Also the specific upper bound for $\lambda$ is not needed anymore. Our option
to present the results in this manner is motivated by the link with the discrete case and the fact that
for the convergence rates in Theorem \ref{conv-KM-rate2} the upper bound will be used again.
\end{remark}

In the following we investigate the convergence rate of the trajectories of the dynamical system \eqref{dyn-syst-KM}. This will be done in terms of the fixed point residual function $t\mapsto\|T(x(t))-x(t)\|$ and  of $t\mapsto\|\dot x(t)\|$. Notice that convergence rates for the discrete iteratively generated algorithm
\eqref{KM-discrete} have been investigated in \cite{corman-yuan, davis-yin, liang-fadili-peyre}.

\begin{theorem}\label{conv-KM-rate2} Let $T:{\cal H}\rightarrow {\cal H}$ be a nonexpansive mapping such that $\fix T\neq\emptyset$,
$\lambda:[0,+\infty)\rightarrow (0,1)$ a Lebesgue measurable function and
$x_0\in {\cal H}$. Suppose that
$$0<\inf_{t\geq 0}\lambda(t)\leq\sup_{t\geq 0}\lambda(t)<1.$$
Let $x:[0,+\infty)\rightarrow{\cal H}$ be the unique strong global solution of \eqref{dyn-syst-KM}. Then for all $t\geq 0$ we have
$$t\|\dot x(t)\|^2\leq t\|T(x(t))-x(t)\|^2\leq \frac{2}{{\ul\tau}} \int_{t/2}^t\lambda(s)(1-\lambda(s))\|T(x(s))-x(s)\|^2 ds,$$
where $\ul\tau=\inf_{t\geq0}\lambda(t)(1-\lambda(t))>0$ and $\lim_{t\rightarrow+\infty}\int_{t/2}^t\lambda(s)(1-\lambda(s))\|T(x(s))-x(s)\|^2 ds=0$.
\end{theorem}

\begin{proof} Define the function $f:[0,+\infty)\rightarrow[0,+\infty)$,
$$f(t)=\int_0^t\lambda(s)(1-\lambda(s))\|T(x(s))-x(s)\|^2 ds.$$
According to \eqref{int-l-T}  we have that $\lim_{t\rightarrow+\infty}f(t)\in\R$.

Since $t\mapsto\frac{1}{2}\|T(x(t))-x(t)\|^2$ is decreasing (see the proof of Theorem \ref{conv-KM}), we have for all $t\geq 0:$
\begin{align*}
\|T(x(t))-x(t)\|^2\int_{t/2}^t\lambda(s)(1-\lambda(s))ds \leq & \int_{t/2}^t\lambda(s)(1-\lambda(s))\|T(x(s))-x(s)\|^2 ds\\
                                                            = & f(t)-f(t/2).
 \end{align*}
Taking into account the definition of $\ul \tau$, we easily derive
$$\frac{{\ul\tau}}{2}t\|T(x(t))-x(t)\|^2\leq \int_{t/2}^t\lambda(s)(1-\lambda(s))\|T(x(s))-x(s)\|^2 ds,$$
and the conclusion follows by using again \eqref{dyn-syst-KM}.
\end{proof}

In the following we investigate a continuous version of the forward-backward algorithm. Let $A:{\cal H}\rightrightarrows {\cal H}$ be a maximally monotone operator, $\beta>0$ and
$B:{\cal H}\rightarrow {\cal H}$ be $\beta$-cocoercive. Consider the dynamical system
\begin{equation}\label{dyn-syst-fb}\left\{
\begin{array}{ll}
\dot x(t)=\lambda(t)\left[J_{\gamma A}\Big(x(t)-\gamma B(x(t))\Big)-x(t)\right]\\
x(0)=x_0.
\end{array}\right.\end{equation}
It is immediate that \eqref{dyn-syst-fb} can be written in the form \eqref{dyn-syst-KM}
for $T=J_{\gamma A}\circ(\id -\gamma B).$ According to \cite[Corollary 23.8 and Remark 4.24(iii)]{bauschke-book}, $J_{\gamma A}$ is $1/2$-cocoercive.
Moreover, by \cite[Proposition 4.33]{bauschke-book}, $\id -\gamma B$ is $\gamma/(2\beta)$-averaged. Combining this with
\cite[Theorem 3(b)]{og-yam}, we derive that $R$ is $1/\delta$-averaged, where $\delta:=\frac{4\beta-\gamma}{2\beta}$. In the next theorem, the statements (i)-(iv) follow from
Remark \ref{average} by noticing that $\fix T=\zer(A+B)$, see \cite[Proposition 25.1(iv)]{bauschke-book}. For (v) and (vi)
we refer to \cite[Theorem 12]{b-c-dyn-KM}.

\begin{theorem}\label{fb-dyn} Let $A:{\cal H}\rightrightarrows {\cal H}$ be a maximally monotone operator, $\beta>0$ and
$B:{\cal H}\rightarrow {\cal H}$ be $\beta$-cocoercive such that $\zer(A+B)\neq\emptyset$. Let $\gamma\in(0,2\beta)$ and set
$\delta=\frac{4\beta-\gamma}{2\beta}$. Let $\lambda:[0,+\infty)\rightarrow [0,\delta]$ be a Lebesgue measurable function and
$x_0\in {\cal H}$. Suppose that either
$$\int_0^{+\infty}\lambda(t)(\delta-\lambda(t))dt=+\infty \ \mbox{or} \ \inf_{t\geq 0}\lambda(t)>0.$$ Let
$x:[0,+\infty)\rightarrow{\cal H}$ be the unique strong global solution of \eqref{dyn-syst-fb}.
Then the following statements are true:

(i) the trajectory $x$ is bounded and $\int_0^{+\infty}\|\dot x(t)\|^2dt<+\infty$;

(ii) $\lim_{t\rightarrow+\infty}\left[J_{\gamma A}\Big(x(t)-\gamma B(x(t))\Big)-x(t)\right]=0$;

(iii) $\lim_{t\rightarrow+\infty}\dot x(t)=0$;

(iv) $x(t)$ converges weakly to a point in $\zer(A+B)$, as $t\rightarrow+\infty$.

\noindent Suppose that $\inf_{t\geq 0}\lambda(t)>0$. Then the following hold:

(v) if $y\in\zer(A+B)$, then $\lim_{t\rightarrow+\infty}B(x(t))=By$ and $B$ is constant on $\zer(A+B)$;

(vi) if $A$ or $B$ is uniformly monotone, then $x(t)$ converges strongly to the unique point in $\zer(A+B)$, as $t\rightarrow+\infty$.
\end{theorem}

\begin{remark} (i) The explicit discretization of \eqref{dyn-syst-fb} with respect to the time variable $t$, with step size $h_n>0$ and initial point $x_0$, yields the following iterative scheme:
$$\frac{x_{n+1}-x_n}{h_n}=\lambda_n\left[J_{\gamma A}\Big(x_n-\gamma Bx_n\Big)-x_n\right] \ \forall n \geq 0.$$

For $h_n=1$ this becomes
\begin{equation}\label{fb-discrete}x_{n+1}=x_n+\lambda_n\left[J_{\gamma A}\Big(x_n-\gamma Bx_n\Big)-x_n\right] \ \forall n \geq 0,\end{equation}
which is the classical forward-backward algorithm for finding the set of zeros of $A+B$
(see \cite[Theorem 25.8]{bauschke-book}). Let us mention that the convergence of \eqref{fb-discrete} is guaranteed under the condition
$\sum_{n \in \N} \lambda_n(\delta-\lambda_n)=+\infty$.

(ii) Let us notice that in case $A+B$ is strongly monotone, then rate of the convergence in (vi) is
exponential, see \cite[Theorem 1]{b-c-dyn-conv-rate}.

(iii) Notice that strong convergence of the trajectory can be induced by Tikhonov regularization.
We refer the reader to \cite{bgms} where the dynamics
\begin{equation}\label{dyn-syst-fb-tikh}\left\{
\begin{array}{ll}
\dot x(t)=\lambda(t)\left[J_{\gamma A}\Big(x(t)-\gamma B(x(t))+\epsilon(t)x(t)\Big)-x(t)\right]\\
x(0)=x_0
\end{array}\right.
\end{equation}
has been investigated under appropriate conditions imposed on the Tikhonov regularization in order
to generate strong convergence of the trajectory towards the minimal norm solution. Notice that
\eqref{dyn-syst-fb-tikh} already appears in \cite[Section 5]{bolte-2003} in case $\lambda(t)=1$, $A$ is the normal cone to
a nonempty, closed convex set $C\subseteq {\cal H}$ (that is, $J_{\gamma A}$ is the projector operator onto $C$) and $B$ is the gradient of a convex and smooth function.
\end{remark}

\begin{remark} Let us mention that in case $A=\partial f$, where $f:{\cal H}\rightarrow\R\cup\{+\infty\}$ is a proper, convex
and lower semicontinuous function defined on a real Hilbert space ${\cal H}$, and for $\lambda(t)=1$ for all $t\geq 0$, the dynamical system
\eqref{dyn-syst-fb} has been studied in \cite{abbas-att-arx14}. Notice that the weak convergence
is obtained in \cite[Theorem 5.2]{abbas-att-arx14} for a constant step-size $\gamma\in(0,4\beta)$.
\end{remark}

\begin{remark} In Theorem \ref{fb-dyn} above related to the dynamical system \eqref{dyn-syst-fb}, the cocoercivity of the operator $B$ plays an important role. Now let us stress the fact that in applications
arising from game theory for example, the operator $(x,y)\to (\nabla_x\phi(x,y),-\nabla_y\phi(x,y))$ is not cocoercive (even in the bilinear case when $\phi(x,y)=\langle x, Ky\rangle$ with K a linear and continuous
operator). Under further conditions imposed on $\phi:{\mathcal H}\times{\mathcal G}\to\R$ (${\mathcal G}$ is another real Hilbert space), like convexity
in $x$, concavity in $y$ and some smoothness conditions, this operator is monotone and Lipschitz. Hence
the desire to consider dynamical systems and numerical schemes also in this case is well motivated.

With respect to the case when $B$ is monotone and Lipschitz, let us mention two types of dynamics considered in the literature. First, a dynamical system of forward-backward-forward-type
\begin{eqnarray}\label{DS0}
\begin{cases}
\begin{aligned}
y(t)&=J_{\gamma A}(x(t) - \gamma B(x(t)))\\
\dot{x}(t)+x(t) & = y(t)+\lambda\left[ B(x(t))-B(y(t))\right] \\
x(0)&=x_0,
\end{aligned}
\end{cases}
\end{eqnarray}
introduced and investigated in \cite{banert-bot}. The time discretization of this dynamics leads to
the well-known Tseng algorithm \cite{Tseng2000}
\begin{eqnarray}\label{DS0-alg}
\begin{cases}
\begin{aligned}
y_n&=J_{\gamma A}(x_n - \gamma B(x_n))\\
x_{n+1} & = y_n+\lambda\left[ B(x_n)-B(y_n)\right],
\end{aligned}
\end{cases}
\end{eqnarray}
which plays an important role also in primal-dual type algorithms for highly structured
monotone inclusion/optimization problems, see for example \cite{combettes-pesquet}. Let us also mention that the dynamics
\eqref{DS0} and several discretizations have been considered in the context of pseudo-monotone variational
inequalities, see \cite{bcv}.

Second, the dynamical system
\begin{eqnarray}\label{cmm}
\begin{cases}
\begin{aligned}
\dot{x}(t)+x(t) &= J_{\gamma A}\left( x(t)- y(t)\right) -\dot{y}(t)\\
y(t)            &= \gamma B(x(t)) \\
x(0)&=x_0,
\end{aligned}
\end{cases}
\end{eqnarray}
has been investigated in \cite{cmt-amo}. By using the notation $z(t) = x(t) + y(t)$, it can be shown that
the dynamics \eqref{cmm} is equivalent to
\begin{equation}\label{dr-3}
  \dot{z}(t)+z(t)
  = \left(\frac{\id + R_{\gamma A} R_{\gamma B}}{2}\right)z(t),
\end{equation}
where $R_{\gamma A}=2J_{\gamma A}-\id$ is the reflected resolvent of $A$. Notice that this is a continuous
version of the Douglas-Rachford algorithm, see \cite{bauschke-book}. Now, \eqref{dr-3}
can be written in the language of \eqref{dyn-syst-KM}, since the reflected resolvent is a nonexpansive mapping.
Further, let us mention that time discretization of \eqref{cmm} leads to the numerical scheme
\begin{equation}
\label{fbr}
x_{n+1} = J_{\gamma A }\bigl(x_n - \gamma B(x_n)\bigr) - \gamma \bigl(B(x_n)-
B(x_{n-1})\bigr),
\end{equation}
which compared to Tseng's algorithm \eqref{DS0-alg} has the advantage that in each iteration it requires only one forward
evaluation of the single valued operator $B$ (see also \cite{mt} for another numerical scheme with this
property). Finally, we mention also \cite{rt} for continuous in time dynamics and numerical schemes
for monotone inclusion problems involving monotone and Lipschitz operators.
\end{remark}

The results presented in this section can be specialized in the context of the convex optimization problem:
\begin{equation}\label{opt-fg}\min_{x\in {\cal H}}\{f(x)+g(x)\}
\end{equation}
where $f:{\cal H}\rightarrow \R\cup\{+\infty\}$ is a proper, convex
and lower semicontinuous function, $\beta>0$ and
$g:{\cal H}\rightarrow \R$ is a convex and (Fr\'{e}chet) differentiable function with $1/\beta$-Lipschitz continuous gradient (which implies that $\nabla g$ is $\beta$-cocoercive, due to the celebrated Baillon-Haddad Theorem, see \cite[Corollary 8.16]{bauschke-book}). Notice that
$$\argmin_{x\in {\cal H}}\{f(x)+g(x)\}=\zer(\partial f+\nabla g)$$
and the system \eqref{dyn-syst-fb} becomes
\begin{equation}\label{dyn-syst-fb-opt}\left\{
\begin{array}{ll}
\dot x(t)=\lambda(t)\left[\prox_{\gamma f}\Big(x(t)-\gamma \nabla g(x(t))\Big)-x(t)\right]\\
x(0)=x_0.
\end{array}\right.\end{equation}

Let us underline the remarkable result that in case $\lambda(t)=1$, the trajectory of \eqref{dyn-syst-fb-opt} converges weakly to a solution of \eqref{opt-fg} with no further restriction on the positive step size $\gamma$. This has been
proved in \cite[Theorem 5.2]{abbas-att-arx14} by using the energy functional
\begin{equation}\label{en-funct}E(t,x^*)=\frac{1}{2\gamma}\|x(t)-x^*\|^2+g(x(t))-g(x^*)-\langle \nabla g(x^*),x(t)-x^*\rangle,\end{equation}
where $x^*$ is a solution of \eqref{opt-fg}.

For the prox-gradient scheme (see ISTA in \cite{BecTeb09}) it is known that we have the rate $O\left(\frac{1}{n}\right)$ for the functions values along the generated iterates. The following theorem can be seen 
as the continuous counterpart of this result. 

\begin{theorem}\label{cont-ista} In the above setting, we consider $\lambda(t)=1$ and
$\frac{\gamma}{\beta}(3+\frac{\gamma}{\beta})\leq 1$.
For $x_0\in{\cal H}$, let $x(\cdot):[0,+\infty)\rightarrow{\cal H}$ be the unique strong global solution of
\eqref{dyn-syst-fb-opt}. Then for all $T>0$ the following holds
\begin{equation}\label{1/T} 0\leq(f+g)\big(\dot x(T)+x(T)\big)-(f+g)(x^*)+\frac{1}{2\gamma}\|\dot x(T)\|^2\leq \frac{1}{2\gamma T}\|x_0-x^*\|^2
\end{equation}
where $x^*$ is a solution of \eqref{opt-fg}.
\end{theorem}

\begin{proof} From the characterization of the proximal operator we have
\begin{equation}\label{from-def-prox}-\frac{1}{\gamma}\dot x(t)-\nabla g(x(t))\in\partial f(\dot x(t)+x(t)),\end{equation}
The convexity of $f$ yields
\begin{equation}\label{f-conv} f(x^*)\geq f\big(\dot x(t)+x(t)\big)+
\left\langle -\frac{1}{\gamma}\dot x(t)-\nabla g(x(t)), x^*-\dot x(t)-x(t)\right\rangle.
\end{equation}
Further, by using the convexity of $g$ and the Descent Lemma (see \cite[Lemma 1.2.3]{nes}) we derive
\begin{align}g(x^*)\geq & \ g(x(t))+\langle\nabla g(x(t)),x^*-x(t)\rangle\nonumber\\
\geq & \ g\big(\dot x(t)+x(t)\big)-\langle\nabla g(x(t)),\dot x(t)\rangle-\frac{1}{2\beta}\|\dot x(t)\|^2+\langle\nabla g(x(t)),x^*-x(t)\rangle\nonumber\\
= &\label{g} \ g\big(\dot x(t)+x(t)\big)+\langle\nabla g(x(t)),x^*-\dot x(t)-x(t)\rangle-\frac{1}{2\beta}\|\dot x(t)\|^2\end{align}

Combining \eqref{f-conv} and \eqref{g} we obtain
\begin{align}(f+g)(x^*)\geq & \ (f+g)\big(\dot x(t)+x(t)\big)-
\frac{1}{\gamma}\langle \dot x(t), x^*-\dot x(t)-x(t)\rangle-\frac{1}{2\beta}\|\dot x(t)\|^2\nonumber\\
= & \ (f+g)\big(\dot x(t)+x(t)\big)+\frac{2-\frac{\gamma}{\beta}}{2\gamma}\|\dot x(t)\|^2+\frac{1}{2\gamma}\frac{d}{dt}\left(\|x(t)-x^*\|^2\right)\nonumber\\
\geq &\label{ineq-obj1} \ (f+g)\big(\dot x(t)+x(t)\big)+\frac{1}{2\gamma}\|\dot x(t)\|^2+\frac{1}{2\gamma}\frac{d}{dt}\left(\|x(t)-x^*\|^2\right).
\end{align}

Let us fix an arbitrary $T>0$. By integration, we derive from \eqref{ineq-obj1}
\begin{equation}\label{ineq-obj2} \frac{1}{2\gamma}\|x(T)-x^*\|^2+
\int_0^T\left[(f+g)\big(\dot x(t)+x(t)\big)-(f+g)(x^*)+\frac{1}{2\gamma}\|\dot x(t)\|^2\right]dt\leq  \frac{1}{2\gamma}\|x_0-x^*\|^2.
\end{equation}

We notice that due to \eqref{dyn-syst-fb-opt}, we have $\dot x(t) = M\big(x(t)\big)$ where $M:{\cal H}\rightarrow{\cal H}$, $M=\prox_{\gamma f}\circ(\id-\gamma\nabla g)-\id $ is a $(2+\frac{\gamma}{\beta})$-Lipschitz continuous operator. Hence, $\dot x$ is locally absolutely
continuous and (see Remark \ref{rem-abs-cont}(b)) $\ddot x$ exists and for almost every $t\geq 0$ one has
\begin{equation}\label{ddot-exist}\|\ddot x(t)\|\leq\left(2+\frac{\gamma}{\beta}\right)\|\dot x(t)\|.\end{equation}

Due to the continuity properties on $[0,T]$ of the trajectory, one has
$$\dot x+x\in L^2([0,T];{\cal H}),$$
$$-\frac{1}{\gamma}\dot x-\nabla g(x)\in L^2([0,T];{\cal H})$$
$$\ddot x+\dot x\in L^2([0,T];{\cal H}).$$
From \eqref{from-def-prox} and \cite[Lemme 4, p. 73]{brezis} (see also \cite[Lemma 3.2]{att-cza-10}) we obtain that the function $t\mapsto f\big(\dot x(t)+x(t)\big)$ is absolutely continuous
and $$\frac{d}{dt}f\big(\dot x(t)+x(t)\big)=\left\langle -\frac{1}{\gamma}\dot x(t)-\nabla g(x(t)),\ddot x(t)+\dot x(t)\right\rangle.$$
Moreover, it holds
$$\frac{d}{dt}g\big(\dot x(t)+x(t)\big)=\left\langle \nabla g\big(\dot x(t)+x(t)\big),\ddot x(t)+\dot x(t)\right\rangle.$$

Summing up the last two equalities we derive
\begin{align}\frac{d}{dt}(f+g)\big(\dot x(t)+x(t)\big) = & \left\langle -\frac{1}{\gamma}\dot x(t)-\nabla g(x(t))+\nabla g\big(\dot x(t)+x(t)\big),
\ddot x(t)+\dot x(t)\right\rangle\nonumber\\
 = & -\frac{1}{2\gamma}\frac{d}{dt}\big(\|\dot x(t)\|^2\big)-\frac{1}{\gamma}\|\dot x(t)\|^2\nonumber\\
 & +\left\langle \nabla g\big(\dot x(t)+x(t)\big)-\nabla g(x(t)),
\ddot x(t)+\dot x(t)\right\rangle\nonumber\\
\label{lip-g}\leq & -\frac{1}{2\gamma}\frac{d}{dt}\big(\|\dot x(t)\|^2\big)-\frac{1}{\gamma}\|\dot x(t)\|^2
 + \frac{1}{\beta}\|\dot x(t)\|\cdot\|\ddot x(t)+\dot x(t)\|\\
 \label{ddot}\leq & -\frac{1}{2\gamma}\frac{d}{dt}\big(\|\dot x(t)\|^2\big)-\frac{1}{\gamma}\|\dot x(t)\|^2
 + \frac{1}{\beta}(3+\frac{\gamma}{\beta})\|\dot x(t)\|^2\\
 = & -\frac{1}{2\gamma}\frac{d}{dt}\big(\|\dot x(t)\|^2\big)-\left[\frac{1}{\gamma}-\frac{1}{\beta}(3+\frac{\gamma}{\beta})\right]\|\dot x(t)\|^2,\nonumber
\end{align}
where in \eqref{lip-g} we used the Lipschitz continuity of $\nabla g$ and in \eqref{ddot} the inequality \eqref{ddot-exist}.
Altogether, we conclude that for almost every $t\geq 0$ we have
\begin{equation}\label{decr-f-conv} \frac{d}{dt}\left[(f+g)\big(\dot x(t)+x(t)\big)+\frac{1}{2\gamma}\|\dot x(t)\|^2\right]+
\left[\frac{1}{\gamma}-\frac{1}{\beta}(3+\frac{\gamma}{\beta})\right]\|\dot x(t)\|^2\leq 0.
\end{equation}
Due to the condition imposed on the step size, we get that $$\frac{d}{dt}\left[(f+g)\big(\dot x(t)+x(t)\big)+\frac{1}{2\gamma}\|\dot x(t)\|^2\right]\leq 0,$$
hence the function $t\mapsto (f+g)\big(\dot x(t)+x(t)\big)-(f+g)(x^*)+\frac{1}{2\gamma}\|\dot x(t)\|^2$ is nonincreasing. This together
with \eqref{ineq-obj2} yields the inequality
\begin{equation*}\frac{1}{2\gamma}\|x(T)-x^*\|^2+T\left[(f+g)\big(\dot x(T)+x(T)\big)-(f+g)(x^*)+\frac{1}{2\gamma}\|\dot x(T)\|^2\right]
\leq \frac{1}{2\gamma }\|x_0-x^*\|^2
\end{equation*}
and the proof is complete.
\end{proof}

\begin{remark} Let us mention now a recent contribution related to structured optimization problems.
For ${\mathcal H}$ and ${\mathcal G}$ real Hilbert spaces, we consider the convex minimization problem
\begin{equation}\label{primal}
\min_{x\in{\mathcal H}}f(x)+h(x)+g(Ax),
\end{equation}
where $f:{\mathcal H}\To\R \cup \{+\infty\}$ and $g:{\mathcal G}\To \R \cup \{+\infty\}$ are proper, convex and lower semicontinuous functions, $h:\mathcal{H}\To\R$ is a convex and Fr\'echet differentiable function with Lipschitz continuous gradient and $A:{\mathcal H}\To{\mathcal G}$ is a continuous linear operator.

Due to the presence of the composition with a linear operator, dual variables appear both when writing the Fenchel dual problem and also in primal-dual numerical schemes existing in the literature for solving highly
structured optimization problems. The following dynamical system introduced in \cite{bcl-dyn-pd} can be seen as
a continuous version of the mentioned numerical schemes for solving \eqref{primal}:
\begin{equation}\label{ADMMdysy-subdiff}
\left\{
\begin{array}{llll}
\dot{x}(t)+x(t)\in\left(\partial f +cA^*A+M_1(t)\right)^{-1}\left(M_1(t)x(t)+cA^*z(t)-A^*y(t)-\nabla h(x(t))\right)\\
\\
\dot{z}(t)+z(t)\in \left(\partial g +cI+M_2(t)\right)^{-1}\left(M_2(t)z(t)+cA(\gamma\dot x(t)+x(t))+y(t)\right)\\
\\
\dot{y}(t)=c A(x(t)+\dot{x}(t))-c (z(t)+\dot{z}(t))\\
\\
x(0)=x^0\in{\mathcal H},\,z(0)=z^0\in{\mathcal G},\,y(0)=y^0 \in{\mathcal G},
\end{array}
\right.
\end{equation}
where $c >0$, $\gamma \in [0,1]$, $A^*:{\mathcal G}\To{\mathcal H}$ is the adjoint operator and $M_1:[0,+\infty)\To S_+(\mathcal{H})$ and $M_2:[0,+\infty)\To S_+(\mathcal{G})$ (which means that for all $t\geq 0$, $M_1(t):\mathcal{H}\to\mathcal{H}$ is linear, continuous, self-adjoint and positive semidefinite).

Let us stress the fact that the analysis of the dynamical system \eqref{ADMMdysy-subdiff} is involved.
It needs more technical results in order to show that the system is well-posed. Moreover, the asymptotic  analysis requires deep machinery in order to derive an appropriate energy functional. We invite the
reader to consult \cite{bcl-dyn-pd} for all these details, just mentioning that under appropriate hypotheses,
$(x(t),z(t),y(t))$ converges weakly to a saddle point of the Lagrangian $l$, defined as
$l:{\mathcal H}\times{\mathcal G}\times{\mathcal G}\To\oR,\,l(x,z,y)=f(x)+h(x)+g(z)+\<y,Ax-z\>.$

There are not too many works in the literature devoted to the solving of optimization problems involving
compositions with linear operators by means of continuous in time dynamics. Let us also mention a recent contribution of Attouch \cite{att}, where a different dynamical system
is proposed in \cite[Section 2.2]{att} for solving block-structured optimization problems with linear
constraint. We mention also \cite{bitt-c-w}, where a dynamical system attached to a more involved optimization problem is investigated and which is a continuous counterpart of the proximal alternating minimization algorithm AMA (see also the work of Tseng on AMA \cite{tseng1991}).  

The dynamics \eqref{ADMMdysy-subdiff} provides through explicit time discretization a numerical algorithm which is a combination of the linearized proximal method of multipliers and the proximal ADMM algorithm.

Indeed, by writing the inclusion in an equivalent form and using the explicit discretization of  with respect to the time variable $t$ and constant step $h_k\equiv 1$ yields the iterative scheme (see \cite[Remark 1]{bcl-dyn-pd})
\begin{equation}\label{ADMMdiscrete}
\left\{
\begin{array}{llll}
x_{n+1}\in\argmin\limits_{x\in\mathcal{H}}\left(f(x)+\<x-x_n,\nabla h(x_n)\>+\frac{c}{2}\left\|Ax-z_n+\frac{y_n}{c}\right\|^2+\frac12\|x-x_n\|^2_{M_1^n}\right)\\
\\
z_{n+1}\in \argmin\limits_{z\in \mathcal{G}}\left(g(z)+\frac{c}{2}\left\|A(\gamma x_{n+1}+(1-\gamma)x_n)-z+\frac{y_n}{c}\right\|^2+\frac12\|z-z_n\|^2_{M_2^n}\right)\\
\\
y_{n+1}=y_n+c(Ax_{n+1}-z_{n+1})
\end{array}
\right.
\end{equation}
where $(M_1^n)_{n \geq 0}$ and $(M_2^n)_{n \geq 0}$ are two operator sequences in $S_+(\mathcal{H})$ and $S_+(\mathcal{G})$, respectively.

The algorithm \eqref{ADMMdiscrete} is a combination of the linearized proximal method of multipliers and the proximal ADMM algorithm.

Indeed, in the case when $\gamma =1$, \eqref{ADMMdiscrete} becomes the proximal ADMM algorithm with variable metrics from \cite{Ban-Bot-Csetnek} (see, also, \cite{b-c-acm}). If, in addition, $h=0$ and the operator sequences $(M_1^k)_{k \geq 0}$ and $(M_2^k)_{k \geq 0}$ are constant, then \eqref{ADMMdiscrete} becomes the proximal ADMM algorithm investigated in \cite[Section 3.2]{s-teb} (see, also, \cite{fpst}). It is known that the proximal ADMM algorithm can be seen as a generalization of the full splitting primal-dual algorithms of Chambolle-Pock (see \cite{ch-pck}) and Condat-Vu (see \cite{condat2013, vu}).

On the other hand, in the case when $\gamma =0$, \eqref{ADMMdiscrete} becomes an extension of the  linearized proximal method of multipliers of Chen-Teboulle (see \cite{chen-teb}, \cite[Algorithm 1]{s-teb}).

A particular choice for the linear maps $M_1$ and $M_2$ transforms \eqref{ADMMdysy-subdiff} into a dynamical system of primal-dual type formulated in the spirit of the full splitting paradigm.
For every $t \in [0,+\infty)$, define $$M_1(t)=\frac{1}{\tau(t)}I-cA^*A \ \mbox{and} \ M_2(t)=0,$$
where $\tau(t)>0$ is such that $c\tau(t)\|A\|^2\le1$.

In this particular setting, the dynamical system \eqref{ADMMdysy-subdiff} can be equivalently written as (see \cite[Remark 2]{bcl-dyn-pd})
\begin{equation}\label{ADMMdysyM1}
\left\{
\begin{array}{llll}
\dot{x}(t)+x(t)=\prox\nolimits_{\tau(t)f}\big((I-c\tau(t)A^*A)x(t)+c\tau(t) A^*z(t)-\tau(t)A^*y(t)-\tau(t)\nabla h(x(t))\big)\\
\\
\dot{y}(t)+y(t)+c(\gamma-1)A\dot{x}(t)=\prox\nolimits_{cg^*}\big(cA(\gamma\dot{x}(t)+x(t))+y(t)\big)\\
\\
\dot{y}(t)=c A(x(t)+\dot{x}(t))-c (z(t)+\dot{z}(t))\\
\\
x(0)=x^0\in{\mathcal H},\,z(0)=z^0\in{\mathcal G},\,y(0)=y^0 \in{\mathcal G},
\end{array}
\right.
\end{equation}
where $g^*:\mathcal G\to\R\cup\{+\infty\}$ is the Fenchel conjugate of the proper, convex and lower semicontinuous 
function $g$. 
Let us also mention that when $h=0$ and $\gamma =1$ the discretization of the dynamical system \eqref{ADMMdysyM1} leads to the primal-dual algorithm proposed by Chambolle and Pock in \cite{ch-pck}.

\end{remark}

\section{Second order dynamical systems}

As in the previous section, we start with a general dynamical system which will be considered then in several special
instances, including the link to optimization problems as well:

\begin{equation}\label{dyn-syst}\left\{
\begin{array}{ll}
\ddot x(t) + \gamma(t)\dot x(t) + \lambda(t)B(x(t))=0\\
x(0)=u_0, \dot x(0)=v_0.
\end{array}\right.\end{equation}
where $u_0,v_0\in {\cal H}$, $B: {\cal H}\rightarrow{\cal H}$ is a $\beta$-cocoercive operator,
$\lambda:[0,+\infty)\rightarrow [0,+\infty)$ is a relaxation function in time and $\gamma:[0,+\infty)\rightarrow [0,+\infty)$ is a continuous damping parameter.

The existence and uniqueness of the trajectory follows (under appropriate conditions imposed on the parameters) from the Cauchy-Lipschitz-Picard Theorem by writing \eqref{dyn-syst} as a first order dynamical system in a product space
(see \cite[Theorem 4]{b-c-sicon2016}).

In order to prove the convergence of the trajectories of \eqref{dyn-syst}, we make the following assumptions on the relaxation function $\lambda$ and
the damping parameter $\gamma$, respectively:
\begin{enumerate}
\item[{\rm (A1)}] $\lambda, \gamma :[0,+\infty)\rightarrow (0,+\infty)$ are locally absolutely continuous and there exists $\theta >0$ such that for almost every $t\in [0, +\infty)$ we have
\begin{equation}\label{h3-g}\dot\gamma(t)\leq 0\leq\dot\lambda(t) \mbox{ and } \frac{\gamma^2(t)}{\lambda(t)}\geq\frac{1+\theta}{\beta}.\end{equation}
\end{enumerate}

Due to Definition \ref{abs-cont} and Remark \ref{rem-abs-cont}(a), $\dot\lambda(t),\dot\gamma(t)$ exists for almost every $t\geq 0$ and
$\dot\lambda,\dot\gamma$ are Lebesgue integrable on each interval $[0,b]$ for $0<b<+\infty$. This combined with
$\dot\gamma(t)\leq 0\leq\dot\lambda(t)$ and the fact that $\lambda,\gamma$ take only positive values yield
the existence of a positive lower bound $\ul\lambda$ for $\lambda$ and of a positive upper bound $\ol\gamma$
for $\gamma$. Furthermore, the second assumption in \eqref{h3-g} provides also a positive upper bound $\ol\lambda$ for $\lambda$ and a positive
lower bound $\ul\gamma$ for $\gamma$.

We would also like to point out that under the conditions considered in (A1) the global version of the Cauchy-Lipschitz-Picard Theorem allows us
to conclude that, for $u_0,v_0\in {\cal H}$, there exists a unique trajectory $x:[0,+\infty)\rightarrow{\cal H}$ which is a
$C^2$-function and which satisfies the first relation in \eqref{dyn-syst} for every $t\in [0,+\infty)$. The considerations we make in the following take into account this fact.

Let us also mention that in case $\gamma(t)=\gamma$ and $\lambda(t)=\lambda$ for every $t\in [0,+\infty)$, where
$\gamma,\lambda>0$, the assumption $(A1)$ becomes $\gamma^2\beta > \lambda$, a condition which has been used
in \cite{att-maing} in relation with the study of the asymptotical convergence of the dynamical system \eqref{dyn-syst}.

We state now the convergence result.
\begin{theorem}\label{conv-th}  Let $B: {\cal H}\rightarrow{\cal H}$ be a $\beta$-cocoercive operator for $\beta>0$ such that
$\zer B:=\{u\in {\cal H}:Bu=0\}\neq\emptyset$, $\lambda,\gamma:[0,+\infty)\rightarrow(0,+\infty)$ be functions fulfilling {\rm (A1)} and
$u_0,v_0\in {\cal H}$. Let $x:[0,+\infty)\rightarrow {\cal H}$ be the unique strong global solution of \eqref{dyn-syst}.
Then the following statements are true:

(i) the trajectory $x$ is bounded and $\dot x,\ddot x,Bx\in L^2([0,+\infty); {\cal H})$;

(ii) $\lim_{t\rightarrow+\infty}\dot x(t)=\lim_{t\rightarrow+\infty}\ddot x(t)=\lim_{t\rightarrow+\infty} B(x(t)=0$;

(iii) $x(t)$ converges weakly to an element in $\zer B$ as $t\rightarrow+\infty$.
\end{theorem}
\begin{proof} (i) Take an arbitrary $x^*\in\zer B$ and consider for every $t \in [0, +\infty)$ the function $h(t)=\frac{1}{2}\|x(t)-x^*\|^2$. We have $\dot h(t)=\langle x(t)-x^*,\dot x(t)\rangle$ and
$\ddot h(t)=\|\dot x(t)\|^2+ \<x(t)-x^* , \ddot x(t)\>$ for every $t \in [0,+\infty)$. Taking into account \eqref{dyn-syst}, we get for every $t \in [0,+\infty)$
\begin{equation}\label{eq-h}\ddot h(t) + \gamma(t)\dot h(t) + \lambda(t)\<x(t)-x^* , B(x(t))\> = \|\dot x(t)\|^2.\end{equation}

The cocoercivity of $B$ and the fact that $Bx^*=0$ yields for every $t \in [0, +\infty)$
$$\ddot h(t) + \gamma(t)\dot h(t) + \beta\lambda(t)\|B(x(t))\|^2  \leq \|\dot x(t)\|^2.$$
Taking again into account \eqref{dyn-syst} one obtains for every $t \in [0, +\infty)$
$$\ddot h(t) + \gamma(t)\dot h(t) + \frac{\beta}{\lambda(t)}\|\ddot x(t)+\gamma(t)\dot x(t)\|^2  \leq \|\dot x(t)\|^2$$
or, equivalently,
\begin{equation*}\ddot h(t) + \gamma(t)\dot h(t)  +
\frac{\beta\gamma(t)}{\lambda(t)}\frac{d}{dt}\big(\|\dot x(t)\|^2\big)  + \left(\frac{\beta\gamma^2(t)}{\lambda(t)}-1\right)||\dot x(t)||^2 +
\frac{\beta}{\lambda(t)}||\ddot x(t)||^2  \leq 0.\end{equation*}
Combining this inequality with
\begin{equation}\label{d-g-h-x}\frac{\gamma(t)}{\lambda(t)}\frac{d}{dt}\big(\|\dot x(t)\|^2\big)  =  \
\frac{d}{dt}\left(\frac{\gamma(t)}{\lambda(t)}\|\dot x(t)\|^2\right)-
\frac{\dot \gamma(t)\lambda(t)-\gamma(t)\dot\lambda(t)}{\lambda^2(t)}\|\dot x(t)\|^2\\
\end{equation}
and \begin{equation}\label{d-g-h}\gamma(t)\dot h(t)=\frac{d}{dt}(\gamma h)(t)-\dot\gamma(t)h(t)\geq \frac{d}{dt}(\gamma h)(t),\end{equation}
it yields for every $t \in [0,+\infty)$
\begin{align*}
& \ddot h(t)  + \frac{d}{dt}(\gamma h)(t) +\\
& \beta\frac{d}{dt}\left(\frac{\gamma(t)}{\lambda(t)}\|\dot x(t)\|^2\right)  + \left(\frac{\beta\gamma^2(t)}{\lambda(t)}+
\beta\frac{-\dot \gamma(t)\lambda(t)+\gamma(t)\dot\lambda(t)}{\lambda^2(t)}-1\right)||\dot x(t)||^2 +
\frac{\beta}{\lambda(t)}||\ddot x(t)||^2  \leq 0.\end{align*}
Now, assumption (A1) delivers for almost every $t \in [0, +\infty)$ the inequality
\begin{equation}\label{ineq-h-renorm-fin}\ddot h(t)  + \frac{d}{dt}(\gamma h)(t) +
\beta\frac{d}{dt}\left(\frac{\gamma(t)}{\lambda(t)}\|\dot x(t)\|^2\right)  + \theta||\dot x(t)||^2 +
\beta\ol\lambda^{-1}\|\ddot x(t)||^2  \leq 0.\end{equation}
This implies that the function $t\mapsto \dot h(t)+\gamma(t) h(t)+\beta\frac{\gamma(t)}{\lambda(t)}\|\dot x(t)\|^2$, which is locally absolutely
continuous, is monotonically decreasing. Hence there exists a real number $M$ such that for every $t \in [0, +\infty)$
\begin{equation}\label{ineq-h-bound} \dot h(t)+\gamma(t) h(t)+\beta\frac{\gamma(t)}{\lambda(t)}\|\dot x(t)\|^2\leq M,
\end{equation}
which yields that for every $t \in [0, +\infty)$
$$\dot h(t)+\ul\gamma h(t)\leq M.$$
By multiplying this inequality with $\exp(\ul\gamma t)$ and then integrating from $0$ to $T$, where $T > 0$, one easily obtains
$$h(T)\leq h(0)\exp(-\ul\gamma T)+\frac{M}{\ul\gamma}(1-\exp(-\ul\gamma T)),$$
thus
\begin{equation}\label{h-bound} h \mbox{ is bounded}\end{equation}
and, consequently,
\begin{equation}\label{x-bound} \mbox{the trajectory }x \mbox{ is bounded}.\end{equation}

On the other hand, from \eqref{ineq-h-bound}, it follows that for every $t \in [0, +\infty)$
$$ \dot h(t)+\beta\ul\gamma\ol\lambda^{-1}\|\dot x(t)\|^2\leq M,$$
hence
$$ \<x(t)-x^*,\dot x(t)\>+\beta\ul\gamma\ol\lambda^{-1}\|\dot x(t)\|^2\leq M.$$
This inequality in combination with \eqref{x-bound} yields
\begin{equation}\label{dotx-bound} \dot x \mbox{ is bounded},\end{equation}
which further implies that
\begin{equation}\label{doth-bound} \dot h \mbox{ is bounded}.\end{equation}

Integrating the inequality \eqref{ineq-h-renorm-fin} we obtain that there exists a real number $N \in \R$ such that for every $t \in [0, +\infty)$
$$\dot h(t)+\gamma(t) h(t)+\beta\frac{\gamma(t)}{\lambda(t)}\|\dot x(t)\|^2+\theta\int_0^t||\dot x(s)||^2ds
+\beta\ol\lambda^{-1}\int_0^t||\ddot x(s)||^2ds\leq N.$$
From here, via \eqref{doth-bound}, we conclude that $\dot x(\cdot),\ddot x(\cdot)\in L^2([0,+\infty); {\cal H})$.
Finally, from \eqref{dyn-syst} and (A1) we deduce $Bx \in L^2([0,+\infty); {\cal H})$ and the proof of (i) is complete.

(ii) For every $t \in [0, +\infty)$ it holds
$$\frac{d}{dt}\left(\frac{1}{2}\|\dot x(t)\|^2\right)=\<\dot x(t),\ddot x(t)\>\leq \frac{1}{2}\|\dot x(t)\|^2+\frac{1}{2}\|\ddot x(t)\|^2$$
and Lemma \ref{fejer-cont2} together with (i) lead to $\lim_{t\rightarrow+\infty}\dot x(t)=0$.

Further, by taking into consideration Remark \ref{rem-abs-cont}(b), for every $t \in [0, +\infty)$ we have
$$\frac{d}{dt}\left(\frac{1}{2}\|B(x(t))\|^2\right)=\<B(x(t)),\frac{d}{dt}(Bx(t))\>\leq \frac{1}{2}\|B(x(t))\|^2+\frac{1}{2\beta^2}\|\dot x(t)\|^2.$$
By using again Lemma \ref{fejer-cont2} and (i) we get
$\lim_{t\rightarrow+\infty}B(x(t))=0$, while the fact that $\lim_{t\rightarrow+\infty}\ddot x(t)=0$ follows from
\eqref{dyn-syst} and (A2).

(iii) We are going to prove that both assumptions in Opial Lemma are fulfilled. The first one concerns the existence
of $\lim_{t\rightarrow +\infty }\|x(t)-x^*\|$. As seen in the proof of part (i), the function
$t\mapsto \dot h(t)+\gamma(t) h(t)+\beta\frac{\gamma(t)}{\lambda(t)}\|\dot x(t)\|^2$ is
monotonically decreasing, thus from (i), (ii) and (A1) we deduce that
$\lim_{t\rightarrow+\infty} \gamma(t) h(t)$ exists and it is a real number.
By taking also into account that $\exists \lim_{t\rightarrow+\infty}\gamma(t)\in(0,\infty)$, we obtain
the existence of $\lim_{t\rightarrow +\infty }\|x(t)-x^*\|$.

We come now to the second assumption of the Opial Lemma. Let $\ol x$ be a weak sequential cluster point of $x$, that is, there exists
a sequence $t_n\rightarrow+\infty$ (as $n\rightarrow+\infty$) such that $(x(t_n))_{n\in\N}$ converges weakly to $\ol x$. Since $B$ is a maximally monotone operator (see for instance \cite[Example 20.28]{bauschke-book}), its graph is sequentially closed with respect to
the weak-strong topology of the product space ${\cal H}\times {\cal H}$. By using also that $\lim_{n\rightarrow+\infty}B(x({t_n}))=0$, we conclude
that $B\ol x=0$, hence $\ol x\in\zer B$ and the proof is complete.
\end{proof}

A standard choice of a cocoercive operator defined on a real Hilbert spaces is $B=\id-T$, where $T: {\cal H}\rightarrow {\cal H}$ is a nonexpansive operator. As it easily follows from the nonexpansiveness of $T$, $B$ is in this case $1/2$-cocoercive. For this particular operator $B$ the dynamical system \eqref{dyn-syst} becomes
\begin{equation}\label{dyn-syst-nonexp}\left\{
\begin{array}{ll}
\ddot x(t) + \gamma(t) \dot x(t) + \lambda(t)\big(x(t)-T(x(t))\big)=0\\
x(0)=u_0, \dot x(0)=v_0.
\end{array}\right.\end{equation}

Theorem \ref{conv-th} gives rise to a corresponding result where the trajectory converges weakly to a fixed point
of $T$, where in (A1) we use the condition $\frac{\gamma^2(t)}{\lambda(t)}\geq 2(1+\theta)$.

\begin{remark}\label{alv-att-nonexp} (i) The analysis can be further extended to the situation when $T$ in \eqref{dyn-syst-nonexp} is an $\alpha$-averaged operator (with $0<\alpha<1$). In this case  we use in (A1) the condition $\frac{\gamma^2(t)}{\lambda(t)}\geq 2\alpha(1+\theta)$.

(ii) In the particular case when $\gamma(t)=\gamma > 0$ for all $t\geq 0$ and
$\lambda(t)=1$ for all $t \in [0, +\infty)$ the dynamical system \eqref{dyn-syst-nonexp} has been studied in \cite[Theorem 3.2]{att-alv}
under the condition $\gamma^2>2$, which is covered by the theory presented above. We would also like to notice that in \cite{alvarez2000} an anisotropic damping has
been considered in the context of approaching the minimization of a smooth convex function via second order dynamical systems.
\end{remark}

Let us turn now our attention to monotone inclusions and consider the dynamics:
\begin{equation}\label{dyn-syst-fb-mon}\left\{
\begin{array}{ll}
\ddot x(t) + \gamma(t)\dot x(t) + \lambda(t)\left[x(t)-J_{\eta A}\Big(x(t)-\eta B(x(t))\Big)\right]=0\\
x(0)=u_0, \dot x(0)=v_0.
\end{array}\right.\end{equation}
Further, in (A1) we use $ \frac{\gamma^2(t)}{\lambda(t)}\geq \frac{2 (1+\theta)}{\delta}$, where $\delta:=\frac{4\beta-\eta}{2\beta}$.

The statements (i)-(iii) in the next result follow from Remark \ref{alv-att-nonexp}(ii) by taking into account that
under the hypotheses, $T=J_{\eta A}\circ(\id -\eta B)$ is $1/\delta$-averaged, see also the comments before
Theorem \ref{fb-dyn}. For (iv) and (v) we invite the reader to consult \cite[Theorem 12]{b-c-sicon2016}.
Moreover, under appropriate hypotheses, the rate of convergence in (v) is exponential, see \cite[Theorem 7]{b-c-dyn-conv-rate}.

\begin{theorem}\label{fb-dyn-sec-ord} Let $A:{\cal H}\rightrightarrows {\cal H}$ be a maximally monotone operator and
$B:{\cal H}\rightarrow {\cal H}$ be a $\beta$-cocoercive operator for $\beta > 0$
such that $\zer(A+B)\neq\emptyset$. Let $\eta\in(0,2\beta)$ and set $\delta:=\frac{4\beta-\eta}{2\beta}$. Let
$\lambda,\gamma:[0,+\infty)\rightarrow(0,+\infty)$ be functions fulfilling {\rm (A1)} as mentioned right after  \eqref{dyn-syst-fb-mon}, $u_0,v_0\in {\cal H}$
and $x:[0,+\infty)\rightarrow {\cal H}$ be the unique strong global solution of \eqref{dyn-syst-fb-mon}.
Then the following statements are true:

(i) the trajectory $x$ is bounded and $\dot x,\ddot x,\big(\id -J_{\eta A}\circ(\id -\eta B)\big)x\in L^2([0,+\infty); {\cal H})$;

(ii) $\lim_{t\rightarrow+\infty}\dot x(t)=\lim_{t\rightarrow+\infty}\ddot x(t)=
\lim_{t\rightarrow+\infty}\big(\id -J_{\eta A}\circ(\id -\eta B)\big)(x(t))=0$;

(iii) $x(t)$ converges weakly to a point in $\zer(A+B)$ as $t\rightarrow+\infty$;

(iv) if $x^*\in\zer(A+B)$, then $B(x(\cdot))-Bx^*\in L^2([0,+\infty); {\cal H})$,
$\lim_{t\rightarrow+\infty}B(x(t))=Bx^*$ and $B$ is constant on $\zer(A+B)$;

(v) if $A$ or $B$ is uniformly monotone, then $x(t)$ converges strongly to the unique point in $\zer(A+B)$ as $t\rightarrow+\infty$.
\end{theorem}

\begin{remark} For second order dynamical systems with vanishing damping we refer the reader to \cite{att-peyp}
where the following differential equation is investigated
\begin{equation}\label{dyn-syst-yos-att-peyp}\left\{
\begin{array}{ll}
\ddot x(t) + \frac{\alpha}{t} \dot x(t) + A_{\lambda(t)}(x(t))=0\\
x(0)=u_0, \dot x(0)=v_0,
\end{array}\right.\end{equation}
$A_{\lambda(t)}: {\cal H}\to  {\cal H}$ being the Yosida regularization of $A$, defined by
$A_{\lambda(t)}=\frac{1}{\lambda(t)}(\id-J_{\lambda(t)A})$. Convergence of the trajectory towards
a zero of the maximally monotone operator $A:{\cal H}\rightrightarrows {\cal H}$ is reported in
\cite{att-peyp} together with convergence rates concerning the velocity and the acceleration. 
Let us also mention the recent contribution \cite{att-lasz-cont-newton} where a Newton-like correction term is used: 
\begin{equation}\label{dyn-syst-yos-att-lasz}\left\{
\begin{array}{ll}
\ddot x(t) + \frac{\alpha}{t} \dot x(t) + \beta\frac{d}{dt}\big(A_{\lambda(t)}(x(t))\big) + A_{\lambda(t)}(x(t))=0\\
x(0)=u_0, \dot x(0)=v_0,
\end{array}\right.\end{equation}
with $\beta > 0$. In case of convex potentials, the correction term becomes the Hessian driven damping which 
is responsable for attenuation of the oscillations of the trajectories and for better convergence rates 
for the gradients along the orbits, see also \cite{att-ch-fad-r-arx2019} and \cite{shi-du-su-jordan2019}. For the algorithmic consequences of \eqref{dyn-syst-yos-att-peyp} and \eqref{dyn-syst-yos-att-lasz} we invite the reader to consult
\cite{att-peyp} and \cite{AL}. 
\end{remark}

In the remaining of this section we turn our attention to optimization problems of the form
\begin{equation*}
\min_{x \in {\cal H}} f(x) + g(x),
\end{equation*}
see \eqref{opt-fg} and the setting described there.

In the following statement we approach the minimizers of $f+g$ via the second order dynamical system
\begin{equation}\label{dyn-syst-fb-opt-sec-ord}\left\{
\begin{array}{ll}
\ddot x(t) + \gamma(t) \dot x(t) + \lambda(t)\left[x(t)-\prox_{\eta f}\Big(x(t)-\eta \nabla g(x(t))\Big)\right]=0\\
x(0)=u_0, \dot x(0)=v_0.
\end{array}\right.\end{equation}

\begin{corollary}\label{fb-dyn-opt} Let $f:{\cal H}\rightarrow\R\cup\{+\infty\}$ by a proper, convex and
lower semicontinuous function and $g:{\cal H}\rightarrow \R$ be a convex and (Fr\'{e}chet) differentiable function with  $\frac{1}{\beta}$-Lipschitz continuous gradient for $\beta >0$
such that $\argmin_{x\in {\cal H}}\{f(x)+g(x)\}\neq\emptyset$. Let $\eta\in(0,2\beta)$ and set $\delta:=\frac{4\beta-\eta}{2\beta}$. Let
$\lambda, \gamma :[0,+\infty)\rightarrow(0,+\infty)$ be functions fulfilling {\rm (A1)} as mentioned right after  \eqref{dyn-syst-fb-mon}, $u_0,v_0\in {\cal H}$ and
$x:[0,+\infty)\rightarrow {\cal H}$ be the unique strong global solution of \eqref{dyn-syst-fb-opt-sec-ord}.
Then the following statements are true:

(i) the trajectory $x$ is bounded and $\dot x,\ddot x,\big(\id -\prox_{\eta f}\circ(\id -\eta \nabla g)\big)x\in L^2([0,+\infty); {\cal H})$;

(ii) $\lim_{t\rightarrow+\infty}\dot x(t)=\lim_{t\rightarrow+\infty}\ddot x(t)=
\lim_{t\rightarrow+\infty}\big(\id -\prox_{\eta f}\circ(\id -\eta \nabla g)\big)(x(t))=0$;

(iii) $x(t)$ converges weakly to a minimizer of $f+g$  as $t\rightarrow+\infty$;

(iv) if $x^*$ is a minimizer of $f+g$, then $\nabla g(x(\cdot))-\nabla g (x^*)\in L^2([0,+\infty); {\cal H})$,
$\lim_{t\rightarrow+\infty}$ $\nabla g(x(t))=\nabla g(x^*)$ and $\nabla g$ is constant on $\argmin_{x\in {\cal H}}\{f(x)+g(x)\}$;

(v) if $f$ or $g$ is uniformly convex, then $x(t)$ converges strongly to the unique minimizer of $f+g$ as $t\rightarrow+\infty$.
\end{corollary}

\begin{remark}\label{alv-att-opt} Consider again the setting in Remark \ref{alv-att-nonexp}(ii), namely, when
$\gamma(t)=\gamma > 0$ for every $t\geq 0$ and $\lambda(t)=1$ for every $t \in [0,+\infty)$.
Furthermore, for $C$ a nonempty, convex, closed subset of ${\cal H}$, let $f=\delta_C$ be the indicator function of $C$, which is defined as being equal to $0$ for $x\in C$ and to $+\infty$, else.
The dynamical system \eqref{dyn-syst-fb-opt} attached in this setting to the minimization of $g$ over $C$ becomes
\begin{equation}\label{dyn-syst-fb-opt-alv-att}\left\{
\begin{array}{ll}
\ddot x(t) + \gamma \dot x(t) + x(t)-\proj_C\big(x(t)-\eta \nabla g(x(t))\big)=0\\
x(0)=u_0, \dot x(0)=v_0.
\end{array}\right.\end{equation}
The asymptotic convergence of the trajectories of \eqref{dyn-syst-fb-opt-alv-att} has been studied in \cite[Theorem 3.1]{att-alv} under the conditions $\gamma^2>2$ and $0<\eta\leq2\beta$.
In this case the corresponding assumption (A1) trivially holds by choosing $\theta$ such that $0<\theta\leq(\gamma^2-2)/2\leq(\delta\gamma^2-2)/2$.
Thus, in order to verify the corresponding (A1) in case  $\lambda(t)=1$ for every $t \in [0,+\infty)$ one needs
to equivalently assume that $\gamma^2>2/\delta$. Since $\delta \geq 1$, this provides a slight improvement over \cite[Theorem 3.1]{att-alv} in what concerns the choice of $\gamma$.
We refer the reader also to \cite{antipin} for an analysis of the convergence rates of trajectories of the dynamical system \eqref{dyn-syst-fb-opt-alv-att} when $g$ is endowed with supplementary properties.
\end{remark}

\begin{remark}\label{opt2} By using the energy functional \eqref{en-funct}, it is possible to go beyond the
condition $0<\eta\leq 2\beta$ considered in Corollary \ref{fb-dyn-opt}. We refer the reader to
\cite[Corollary 16]{b-c-sicon2016} for more details. Moreover, in \cite[Remark 17]{b-c-sicon2016} it has been pointed out
that in case $\gamma(t)=\gamma >0$ for every $t\geq 0$ and $\lambda(t)=1$ for every $t \in [0,+\infty)$,
the only condition imposed on the parameters is $\gamma^2>\frac{\eta}{\beta}+1$.
In other words, this allows in this particular setting a more relaxed choice
for the parameters, beyond the standard assumptions $0<\eta\leq 2\beta$ and $\gamma^2>2$
considered in \cite{att-alv}.
\end{remark}

\begin{remark}\label{discr} The explicit discretization of \eqref{dyn-syst-fb-opt-sec-ord} with respect to the time variable $t$, with step size $h_n > 0$, relaxation variable $\lambda_n >0$,
damping variable $\gamma_n > 0$ and initial points $x_0:= u_0$ and $x_1:= v_0$ yields the following iterative scheme
$$\frac{x_{n+1} - 2x_n + x_{n-1}}{h_n^2} + \gamma_n \frac{x_{n+1}-x_n}{h_n} = \lambda_n\left[\prox\nolimits_{\eta f}\Big(x_n-\eta \nabla g(x_n)\Big) - x_n\right] \ \forall n \geq 1.$$
For $h_n=1$ this becomes
\begin{equation}\label{eqdiscr}
x_{n+1} = \left(1- \frac{\lambda_n}{1+\gamma_n} \right)x_n + \frac{\lambda_n}{1+\gamma_n} \prox\nolimits_{\eta f}\Big(x_n-\eta \nabla g(x_n)\Big) + \frac{\lambda_n}{1+\gamma_n}(x_{n} - x_{n-1})
\ \forall n \geq 1,
\end{equation}
which is a relaxed forward-backward algorithm for minimizing $f+g$ with inertial effects. In order to investigate the convergence properties of the above iterative scheme, it is natural to assume that $(\gamma_n)_{n\geq 1}$ is nonincreasing, $(\lambda_n)_{n\geq 1}$ is nondecreasing, and there exists a lower bound for $(\gamma_n^2/{\lambda_n})_{n\geq 1}$. The control of the inertial term by means of the variable
parameters $\lambda_n$ and $\gamma_n$ could increase the speed of convergence of the algorithm \eqref{eqdiscr}. Making use of the
the sequence $(\lambda_n)_{n\geq 1}$ in \eqref{eqdiscr}, one obtains a relaxed forward-backward scheme, where the relaxation is usually considered in the literature in order to achieve more freedom in the choice of the parameters involved in the numerical scheme. We refer the reader to \cite{att-cab} where this fact has
been underlined in the context of monotone inclusion problems.
\end{remark}

\begin{remark} In case $f$ is vanishing and $\eta \lambda(t)=1$, the dynamical system \eqref{dyn-syst-fb-opt-sec-ord} becomes
 \begin{equation}\label{dyn-syst-opt-sec-ord-su-boyd-candes}\left\{
\begin{array}{ll}
\ddot x(t) + \gamma(t) \dot x(t) + \nabla g(x(t))=0\\
x(0)=u_0, \dot x(0)=v_0.
\end{array}\right.\end{equation}
There is an intensive research in the literature mainly devoted to the vanishing damping of the form
$\gamma(t)=\frac{\alpha}{t}$, with $\alpha >0$. This is due to the fact that in case $\alpha >3$ this provides
a fast convergence result in terms of the objective function of order $o\left(\frac{1}{t^2}\right)$.
The starting point of the investigations was the work of Su, Boyd and Cand\`{e}s \cite{su-boyd-candes}, followed by relevant contributions of Attouch and his co-workers (see for example
\cite{att-ch-r-esaim2019, att-c-p-r-math-pr2018} and the references therein). This is related to the accelerated gradient method in the sense of Nesterov
 $$\left\{
\begin{array}{ll}
y_n &= \  \ x_n+\frac{n-1}{n+\alpha-1}(x_n-x_{n-1})\\
x_{n+1}&=  \ \ y_n-\gamma \nabla g(y_n)
\end{array}\right.$$
where inertial terms by means of iterates accelerate the convergence behaviour of the numerical scheme.
Let us mention also the recent contributions \cite{att-ch-fad-r-arx2019, shi-du-su-jordan2019} where
inertial terms induced by the gradient of $g$ are also considered, which are coming from the discretization of
Hessian driven damping terms, see also \cite{att-p-r-jde2016} and \cite{bcl-tikhonov}.
\end{remark}

\section{First order dynamics for nonconvex optimization problems}

In this section we approach the solving of the optimization problem
\begin{equation}\label{intr-opt-pb}  \ \inf_{x\in\R^n}[f(x)+g(x)], \end{equation}
where $f:\R^n\to \R\cup\{+\infty\}$ is a proper, convex, lower semicontinuous function and
$g:\R^n\to \R$ a (possibly nonconvex) Fr\'{e}chet differentiable function with $\beta$-Lipschitz continuous gradient for
$\beta \geq 0$, i.e., $\|\nabla g(x) - \nabla g(y) \| \leq \beta\|x-y\| \ \forall x,y \in \R^n$, by associating to it the
implicit dynamical system
\begin{equation}\label{intr-dyn-syst}\left\{
\begin{array}{ll}
\dot x(t) +  x(t)=\prox_{\eta f}\Big(x(t)-\eta\nabla g(x(t))\Big)\\
x(0)=x_0,
\end{array}\right.\end{equation}
where $\eta>0$ and $x_0\in\R^n$ is chosen arbitrary (we consider the euclidean space $\R^n$).

\begin{lemma}\label{l-decr} Suppose that $f+g$ is bounded from below and $\eta>0$ fulfills the inequality
\begin{equation}\label{eta-beta}\eta\beta(3+\eta\beta)<1.\end{equation} For $x_0\in\R^n$, let  $x \in C^1([0,+\infty), \R^n)$ be the unique global solution of
\eqref{intr-dyn-syst}. Then the following statements hold:
\begin{enumerate}
                                            \item [(a)] $\dot x\in L^2([0,+\infty);\R^n)$ and $\lim_{t\rightarrow+\infty}\dot x(t)=0$;
                                            \item [(b)] $\exists\lim_{t\rightarrow+\infty}(f+g)\big(\dot x(t)+x(t)\big)\in\R$.
                                           \end{enumerate}
\end{lemma}

\begin{proof} Let us start by noticing that $\dot x$ is locally absolutely
continuous, hence $\ddot x$ exists and for almost every $t \in [0,+\infty)$ one has
\begin{equation}\label{ddot-exist-f-nonc}\|\ddot x(t)\|\leq(2+\eta\beta)\|\dot x(t)\|.\end{equation}

We fix an arbitrary $T>0$. Due to the continuity properties of the trajectory on $[0,T]$, \eqref{ddot-exist} and the Lipschitz continuity of $\nabla g$, one has
$$x, \dot x, \ddot x, \nabla g(x) \in L^2([0,T];\R^n).$$
Further, from the characterization of the proximal point operator we have
\begin{equation}\label{from-def-prox-f-nonc}-\frac{1}{\eta}\dot x(t)-\nabla g(x(t))\in\partial f(\dot x(t)+x(t)) \ \forall t \in [0,+\infty).\end{equation}
Like in the proof of Theorem \ref{cont-ista} we obtain
$$\frac{d}{dt}f\big(\dot x(t)+x(t)\big)=\left\langle -\frac{1}{\eta}\dot x(t)-\nabla g(x(t)),\ddot x(t)+\dot x(t)\right\rangle,$$
$$\frac{d}{dt}g\big(\dot x(t)+x(t)\big)=\left\langle \nabla g\big(\dot x(t)+x(t)\big),\ddot x(t)+\dot x(t)\right\rangle$$
and for almost every $t\in [0,T]$ we have
\begin{equation}\label{decr-f} \frac{d}{dt}\left[(f+g)\big(\dot x(t)+x(t)\big)+\frac{1}{2\eta}\|\dot x(t)\|^2\right]+
\left[\frac{1}{\eta}-\beta(3+\eta\beta)\right]\|\dot x(t)\|^2\leq 0.
\end{equation}
By integration we get
\begin{align}\label{integ}
& (f+g)\big(\dot x(T)+x(T)\big)+\frac{1}{2\eta}\|\dot x(T)\|^2  + \left[\frac{1}{\eta}-\beta(3+\eta\beta)\right] \int_{0}^T \|\dot x(t)\|^2dt \leq \nonumber \\
& (f+g)\big(\dot x(0)+x_0\big)+\frac{1}{2\eta}\|\dot x(0)\|^2.
\end{align}
By using  \eqref{eta-beta} and the fact that $f+g$ is bounded from below and by taking into account that $T > 0$ has been arbitrarily chosen, we obtain
\begin{equation}\label{dot-l2}\dot x\in L^2([0,+\infty);\R^n).      \end{equation}
Due to \eqref{ddot-exist-f-nonc}, this further implies
\begin{equation}\label{ddot-l2}
\ddot x\in L^2([0,+\infty);\R^n).
\end{equation}
Furthermore, for almost every $t \in [0,+\infty)$ we have
$$\frac{d}{dt}\big(\|\dot x(t)\|^2\big)=2\langle \dot x(t),\ddot x(t)\rangle\leq \|\dot x(t)\|^2+\|\ddot x(t)\|^2.$$
By applying Lemma \ref{fejer-cont2}, it follows that $\lim_{t\rightarrow+\infty}\dot x(t)=0$ and the proof of (a) is complete.
From \eqref{decr-f}, \eqref{eta-beta} and by using that $T >0$ has been arbitrarily chosen, we get
$$\frac{d}{dt}\left[(f+g)\big(\dot x(t)+x(t)\big)+\frac{1}{2\eta}\|\dot x(t)\|^2\right] \leq 0$$
for almost every $t \in [0,+\infty)$. From Lemma \ref{fejer-cont1} it follows that
$$\lim_{t\rightarrow+\infty}\left[(f+g)\big(\dot x(t)+x(t)\big)+\frac{1}{2\eta}\|\dot x(t)\|^2\right]$$
exists and it is a real number, hence from $\lim_{t\rightarrow+\infty}\dot x(t)=0$ the conclusion follows.
\end{proof}

For the following generalized subdifferential notions and their basic properties we refer to \cite{boris-carte, rock-wets}.
Let $h:\R^n\rightarrow \R\cup\{+\infty\}$ be a proper and lower semicontinuous function. If $x\in\dom h$, we consider the {\it Fr\'{e}chet (viscosity)
subdifferential} of $h$ at $x$ as the set $$\hat{\partial}h(x)= \left \{v\in\R^n: \liminf_{y\rightarrow x}\frac{h(y)-h(x)-\<v,y-x\>}{\|y-x\|}\geq 0 \right \}.$$ For
$x\notin\dom h$ we set $\hat{\partial}h(x):=\emptyset$. The {\it limiting (Mordukhovich) subdifferential} is defined at $x\in \dom h$ by
$$\partial_L h(x)=\{v\in\R^n:\exists x_k\rightarrow x,h(x_k)\rightarrow h(x)\mbox{ and }\exists v_k\in\hat{\partial}h(x_k),v_k\rightarrow v \mbox{ as }k\rightarrow+\infty\},$$
while for $x \notin \dom h$, one takes $\partial_L h(x) :=\emptyset$. Therefore  $\hat\partial h(x)\subseteq\partial_L h(x)$ for each $x\in\R^n$.

Notice that in case $h$ is convex, these subdifferential notions coincide with the {\it convex subdifferential}, thus
$\hat\partial h(x)=\partial_L h(x)=\{v\in\R^n:h(y)\geq h(x)+\<v,y-x\> \ \forall y\in \R^n\}$ for all $x\in\R^n$.

The graph of the limiting subdifferential fulfills the following closedness criterion: if $(x_k)_{k\in\N}$ and $(v_k)_{k\in\N}$ are sequences in $\R^n$ such that
$v_k\in\partial_L h(x_k)$ for all $k\in\N$, $(x_k,v_k)\rightarrow (x,v)$ and $h(x_k)\rightarrow h(x)$ as $k\rightarrow+\infty$, then $v\in\partial_L h(x)$.

The Fermat rule reads in this nonsmooth setting as follows: if $x\in\R^n$ is a local minimizer of $h$, then $0\in\partial_L h(x)$. We denote by
$$\crit(h)=\{x\in\R^n: 0\in\partial_L h(x)\}$$ the set of {\it (limiting)-critical points} of $h$.

We define the limit set of the trajectory $x$ as
$$\omega (x)=\{\ol x\in\R^n:\exists t_k\rightarrow+\infty \mbox{ such that }x(t_k)\rightarrow\ol x \mbox{ as }k\rightarrow+\infty\}.$$

\begin{lemma}\label{l-lim-crit-f} Suppose that $f+g$ is bounded from below and $\eta>0$ fulfills the inequality \eqref{eta-beta}.
For $x_0\in\R^n$, let  $x \in C^1([0,+\infty), \R^n)$ be the unique global solution of
\eqref{intr-dyn-syst}. Then $$\omega(x)\subseteq \crit (f+g).$$
\end{lemma}

\begin{proof} Let $\ol x\in\omega (x)$ and $t_k\rightarrow+\infty \mbox{ be such that }x(t_k)\rightarrow\ol x
\mbox{ as }k\rightarrow+\infty.$ From the characterization of the prox operator we have
\begin{align}-\frac{1}{\eta}\dot x(t_k)-\nabla g(x(t_k))+\nabla g\big(\dot x(t_k)+x(t_k)\big)\in & \ \partial f\big(\dot x(t_k)+x(t_k)\big)+
\nabla g\big(\dot x(t_k)+x(t_k)\big)\nonumber\\\label{incl-tk} = & \ \partial_L (f+g)\big(\dot x(t_k)+x(t_k)\big) \ \forall k \in \N.
\end{align}

Lemma \ref{l-decr}(a) and the Lipschitz continuity of $\nabla g$ ensure that
\begin{equation}\label{bor1} -\frac{1}{\eta}\dot x(t_k)-\nabla g(x(t_k))+\nabla g\big(\dot x(t_k)+x(t_k)\big)\rightarrow 0 \mbox{ as }k\rightarrow+\infty
\end{equation}
and
\begin{equation}\label{bor2} \dot x(t_k)+x(t_k)\rightarrow \ol x \mbox{ as }k\rightarrow+\infty.
\end{equation}

We claim that \begin{equation}\label{bor3} \lim_{k\rightarrow+\infty}(f+g)\big(\dot x(t_k)+x(t_k)\big)=(f+g)(\ol x).\end{equation}
Due to the lower semicontinuity of $f$ it holds
\begin{equation}\label{from-f-lsc}\liminf_{k\rightarrow+\infty}f\big(\dot x(t_k)+x(t_k)\big)\geq f(\ol x).\end{equation}

Further, since \begin{align*} \dot x(t_k)+x(t_k)= & \argmin_{u\in\R^n}\left[f(u)+\frac{1}{2\eta}\left\|u-\big(x(t_k)-\eta\nabla g(x(t_k))\big)\right\|^2\right]\\
                                      = & \argmin_{u\in\R^n}\left[f(u)+\frac{1}{2\eta}\|u-x(t_k)\|^2+\langle u-x(t_k),\nabla g(x(t_k))\rangle\right] \end{align*}
we have the inequality
\begin{align*}
 & \ f\big(\dot x(t_k)+x(t_k)\big)+\frac{1}{2\eta}\|\dot x(t_k)\|^2+\langle \dot x(t_k),\nabla g(x(t_k))\rangle\\
\leq & \ f(\ol x)+\frac{1}{2\eta}\|\ol x-x(t_k)\|^2+\langle \ol x-x(t_k), \nabla g(x(t_k))\rangle \ \forall k \in \N.
\end{align*}

Taking the limit as $k\rightarrow+\infty$ we derive by using again Lemma \ref{l-decr}(a) that
\begin{equation*}\limsup_{k\rightarrow+\infty}f\big(\dot x(t_k)+x(t_k)\big)\leq f(\ol x),\end{equation*}
which combined with \eqref{from-f-lsc} implies
\begin{equation*}\lim_{k\rightarrow+\infty}f\big(\dot x(t_k)+x(t_k)\big)= f(\ol x).\end{equation*}
By using \eqref{bor2}  and the continuity of $g$ we conclude that \eqref{bor3} is true.

Altogether, from \eqref{incl-tk}, \eqref{bor1}, \eqref{bor2}, \eqref{bor3} and the closedness criteria of the limiting subdifferential we
obtain $0\in\partial_L (f+g)(\ol x)$ and the proof is complete.
\end{proof}

\begin{lemma}\label{l-h123} Suppose that $f+g$ is bounded from below and $\eta>0$ fulfills the inequality \eqref{eta-beta}.
For $x_0\in\R^n$, let  $x \in C^1([0,+\infty), \R^n)$ be the unique global solution of
\eqref{intr-dyn-syst} and consider the function
$$H:\R^n\times\R^n\to\R\cup\{+\infty\},\, H(u,v)=(f+g)(u)+\frac{1}{2\eta}\|u-v\|^2.$$
Then the following statements are true:
\begin{itemize}
\item[($H_1$)] for almost every $t\in [0,+\infty)$ it holds
$$\frac{d}{dt}H\big(\dot x(t)+x(t),x(t)\big)\leq -\left[\frac{1}{\eta}-(3+\eta\beta)\beta\right]\|\dot x(t)\|^2\leq 0$$  and
$$\exists\lim_{t\rightarrow +\infty}H\big(\dot x(t)+x(t),x(t)\big)\in\R;$$
\item[($H_2$)] for every $t\in [0,+\infty)$ it holds  $$z(t):=\left(-\nabla g(x(t))+\nabla g\big(\dot x(t)+x(t)\big),-\frac{1}{\eta}\dot x(t)\right)\in\partial_L H\big(\dot x(t)+x(t),x(t)\big)$$
and $$\|z(t)\|\leq \left(\beta+\frac{1}{\eta}\right)\|\dot x(t)\|;$$
\item[($H_3$)] for $\ol x\in\omega (x)$ and $t_k\rightarrow+\infty$ such that $x(t_k)\rightarrow\ol x$ as $k\rightarrow+\infty$, we have
$H\big(\dot x(t_k)+x(t_k),x(t_k)\big)\rightarrow H(\ol x,\ol x)$ as $k\rightarrow+\infty$.
\end{itemize}
\end{lemma}

\begin{proof} (H1) follows from Lemma \ref{l-decr}. The first statement in (H2) follows from the characterization of the proximal operator and the relation
\begin{equation}\label{H-subdiff}\partial_L H(u,v)=\big(\partial_L (f+g)(u)+\eta^{-1}(u-v)\big)\times \{\eta^{-1}(v-u)\} \ \forall (u,v)\in\R^n\times\R^n,\end{equation}
while the second one is a consequence of the Lipschitz continuity of $\nabla g$. Finally, (H3) has been shown as intermediate step in the proof of Lemma \ref{l-lim-crit-f}.
\end{proof}

\begin{lemma}\label{l} Assume that the hypotheses of Lemma \ref{l-h123} hold.
Suppose that  $x$ is bounded. Then the following statements are true:
\begin{itemize}
\item[(a)] $\omega(\dot x+x,x)\subseteq \crit(H)=\{(u,u)\in\R^n\times\R^n:u\in \crit(f+g)\}$;
\item[(b)] $\lim_{t\to+\infty}\dist\Big(\big(\dot x(t)+x(t),x(t)\big),\omega\big(\dot x + x,x\big)\Big)=0$;
\item[(c)] $\omega\big(\dot x+x,x\big)$ is nonempty, compact and connected;
\item[(d)] $H$ is finite and constant on $\omega\big(\dot x+x,x\big).$
\end{itemize}
\end{lemma}

\begin{proof} (a), (b) and (d) are direct consequences Lemma \ref{l-decr}, Lemma \ref{l-lim-crit-f} and Lemma \ref{l-h123}.

Finally, (c) is a classical result from \cite{haraux}. We also refer the reader  to the proof of Theorem 4.1 in \cite{alv-att-bolte-red}, where it is shown that the properties of $\omega(x)$ of being nonempty, compact and connected are generic for bounded trajectories fulfilling  $\lim_{t\rightarrow+\infty}{\dot x(t)}=0$.
\end{proof}

\begin{remark}\label{cond-x-bound}
Suppose that $\eta>0$ fulfills the inequality \eqref{eta-beta} and $f+g$ is coercive, that is $$\lim_{\|u\|\rightarrow+\infty}(f+g)(u)=+\infty.$$
For $x_0\in\R^n$, let  $x \in C^1([0,+\infty), \R^n)$ be the unique global solution of
\eqref{intr-dyn-syst}. Then $f+g$ is bounded from below and
$x$ is bounded.

Indeed, since $f+g$ is a proper, lower semicontinuous and coercive function, it follows that
$\inf_{u\in\R^n}[f(u)+g(u)]$ is finite and the infimum is attained. Hence $f+g$ is bounded from below. On the other hand, from \eqref{integ} it follows
\begin{align*}(f+g)\big(\dot x(T)+x(T)\big)\leq & \ (f+g)\big(\dot x(T)+x(T)\big)+\frac{1}{2\eta}\|\dot x(T)\|^2\\
  \leq & \ (f+g)\big(\dot x(0)+x_0)\big)+\frac{1}{2\eta}\|\dot x(0)\|^2 \ \forall T \geq 0.
\end{align*}
Since the lower level sets of $f+g$ are bounded, the above inequality yields the boundedness of $\dot x+x$, which
combined with $\lim_{t\rightarrow+\infty}\dot x(t)=0$ delivers the boundedness of $x$.
\end{remark}

The class of functions
satisfying the {\it Kurdyka-\L{}ojasiewicz property} plays
a crucial role in the asymptotic analysis of the dynamical system  \eqref{intr-dyn-syst}. For $\eta\in(0,+\infty]$, we denote by $\Theta_{\eta}$ the class of concave and continuous functions
$\varphi:[0,\eta)\rightarrow [0,+\infty)$ such that $\varphi(0)=0$, $\varphi$ is continuously differentiable on $(0,\eta)$, continuous at $0$ and $\varphi'(s)>0$ for all
$s\in(0, \eta)$. In the following definition (see \cite{att-b-red-soub2010, b-sab-teb}) we use also the {\it distance function} to a set, defined for $A\subseteq\R^n$ as $\dist(x,A)=\inf_{y\in A}\|x-y\|$
for all $x\in\R^n$.

\begin{definition}\label{KL-property} \rm({\it Kurdyka-\L{}ojasiewicz property}) Let $h:\R^n\rightarrow\R\cup\{+\infty\}$ be a proper and lower semicontinuous
function. We say that $h$ satisfies the {\it Kurdyka-\L{}ojasiewicz (KL) property} at $\ol x\in \dom\partial_L h=\{x\in\R^n:\partial_L h(x)\neq\emptyset\}$, if there exist $\eta \in(0,+\infty]$, a neighborhood $U$ of $\ol x$ and a function $\varphi\in \Theta_{\eta}$ such that for all $x$ in the
intersection
$$U\cap \{x\in\R^n: h(\ol x)<h(x)<h(\ol x)+\eta\}$$ the following inequality holds
$$\varphi'(h(x)-h(\ol x))\dist(0,\partial_L h(x))\geq 1.$$
If $h$ satisfies the KL property at each point in $\dom\partial h$, then $h$ is called {\it KL function}.
\end{definition}

The origins of this notion go back to the pioneering work of \L{}ojasiewicz \cite{lojasiewicz1963}, where it is proved that for a real-analytic function
$h:\R^n\rightarrow\R$ and a critical point $\ol x\in\R^n$ (that is $\nabla h(\ol x)=0$), there exists $\theta\in[1/2,1)$ such that the function
$|h-h(\ol x)|^{\theta}\|\nabla h\|^{-1}$ is bounded around $\ol x$. This corresponds to the situation when $\varphi(s)=Cs^{1-\theta}$, where
$C>0$. The result of
\L{}ojasiewicz allows the interpretation of the KL property as a re-parametrization of the function values in order to avoid flatness around the
critical points. Kurdyka \cite{kurdyka1998} extended this property to differentiable functions definable in o-minimal structures.
Further extensions to the non-smooth setting can be found in \cite{b-d-l2006, att-b-red-soub2010, b-d-l-s2007, b-d-l-m2010}.

One of the remarkable properties of the KL functions is their ubiquity in applications (see \cite{b-sab-teb}). To the class of KL functions belong semi-algebraic, real sub-analytic, semiconvex, uniformly convex and
convex functions satisfying a growth condition. We refer the reader to
\cite{b-d-l2006, att-b-red-soub2010, b-d-l-m2010, b-sab-teb, b-d-l-s2007, att-b-sv2013, attouch-bolte2009} and the references therein for more on KL functions and illustrating examples.

We come now to the main result of this section.

\begin{theorem}\label{conv-kl} Suppose that $f+g$ is bounded from below and $\eta>0$ fulfills the inequality \eqref{eta-beta}. For $x_0\in\R^n$, let  $x \in C^1([0,+\infty), \R^n)$ be the unique global solution of \eqref{intr-dyn-syst} and consider the function
$$H:\R^n\times\R^n\to\R\cup\{+\infty\},\, H(u,v)=(f+g)(u)+\frac{1}{2\eta}\|u-v\|^2.$$
Suppose that  $x$ is bounded and $H$ is a KL function. Then the following statements are true:
\begin{itemize}\item[(a)] $\dot x\in L^1([0,+\infty);\R^n)$;
\item[(b)] there exists $\ol x\in\crit(f+g)$ such that $\lim_{t\rightarrow+\infty}x(t)=\ol x$.
\end{itemize}
\end{theorem}

\begin{proof} According to Lemma \ref{l}, we can choose an element $\ol x\in\crit (f+g)$ such that
$(\ol x,\ol x)\in \omega (\dot x+x,x)$. According to Lemma \ref{l-h123}, it follows that
$$\lim_{t\rightarrow+\infty}H\big(\dot x(t)+x(t),x(t)\big)=H(\ol x,\ol x).$$

We treat the following two cases separately.

I. There exists $\ol t\geq 0$ such that $$H\big(\dot x(\ol t)+x(\ol t),x(\ol t)\big)=H(\ol x,\ol x).$$ Since from
Lemma \ref{l-h123}(H1) we have $$\frac{d}{dt}H\big(\dot x(t)+x(t),x(t)\big) \leq 0 \ \forall t \in [0,+\infty),$$  we obtain  for every $t\geq \ol t$ that
$$H\big(\dot x(t)+x(t),x(t)\big)\leq H\big(\dot x(\ol t)+x(\ol t),x(\ol t)\big)=H(\ol x,\ol x).$$ Thus $H\big(\dot x(t)+x(t),x(t)\big)=H(\ol x,\ol x)$ for every $t\geq \ol t$. This yields by Lemma \ref{l-h123}(H1) that
$\dot x(t)=0$ for almost every $t \in [\ol t, +\infty)$, hence $x$ is constant on $[\ol t,+\infty)$ and the conclusion follows.

II. For every $t\geq 0$ it holds $H\big(\dot x(t)+x(t),x(t)\big)>H(\ol x,\ol x).$ Take $\Omega=\omega(\dot x+x,x)$.

In virtue of Lemma \ref{l}(c) and (d) and since $H$ is a KL function, there exist positive numbers $\epsilon$ and $\eta$ and
a concave function $\varphi\in\Theta_{\eta}$ such that for all
\begin{align}\label{int-H}
(x,y)\in & \{(u,v)\in\R^n\times\R^n: \dist((u,v),\Omega)<\epsilon\} \nonumber \\
 & \cap\{(u,v)\in\R^n\times\R^n:H(\ol x,\ol x)<H(u,v)<H(\ol x,\ol x)+\eta\}\end{align}
one has
\begin{equation}\label{ineq-H}\varphi'(H(x,y)-H(\ol x,\ol x))\dist((0,0),\partial_L H(x,y))\ge 1.\end{equation}

Let $t_1\geq 0$ be such that $H\big(\dot x(t)+x(t),x(t)\big)<H(\ol x,\ol x)+\eta$ for all $t\geq t_1$. Since
$\lim_{t\to+\infty}\dist\Big(\big(\dot x(t)+x(t),x(t)\big),\Omega\Big)=0$, there exists $t_2\geq 0$ such that
$\dist\Big(\big(\dot x(t)+x(t),x(t)\big),\Omega\Big)<\epsilon$ for all $t\geq t_2$. Hence for all $t\geq T:=\max\{t_1,t_2\}$,
$\big(\dot x(t)+x(t),x(t)\big)$ belongs to the intersection in \eqref{int-H}. Thus, according to \eqref{ineq-H}, for every $t\geq T$ we have
\begin{equation}\label{ineq-Ht1}\varphi'\Big(H\big(\dot x(t)+x(t),x(t)\big)-H(\ol x,\ol x)\Big)
\dist\Big((0,0),\partial_L H\big(\dot x(t)+x(t),x(t)\big)\Big)\ge 1.\end{equation}
By applying Lemma \ref{l-h123}(H2) we obtain for almost every $t \in [T, +\infty)$
\begin{equation}\label{ineq-Ht2}(\beta+\eta^{-1})\|\dot x(t)\|\varphi'\Big(H\big(\dot x(t)+x(t),x(t)\big)-H(\ol x,\ol x)\Big)
\ge 1.\end{equation}
From here, by using Lemma \ref{l-h123}(H1) and that $\varphi'>0$ and
\begin{align*}
& \frac{d}{dt}\varphi\Big(H\big(\dot x(t)+x(t),x(t)\big)-H(\ol x,\ol x)\Big)=\\
& \varphi'\Big(H\big(\dot x(t)+x(t),x(t)\big)-H(\ol x,\ol x)\Big)\frac{d}{dt}H\big(\dot x(t)+x(t),x(t)\big),
\end{align*}
we deduce that for almost every $t \in [T, +\infty)$ it holds
\begin{equation}\label{ineq-pt-conv-r} \frac{d}{dt}\varphi\Big(H\big(\dot x(t)+x(t),x(t)\big)-H(\ol x,\ol x)\Big)\leq
-\left(\beta+\eta^{-1}\right)^{-1}\left[\frac{1}{\eta}-(3+\eta\beta)\beta\right]\|\dot x(t)\|.\end{equation}
Since $\varphi$ is bounded from below, by taking into account \eqref{eta-beta}, it follows  $\dot x\in L^1([0,+\infty);\R^n)$. From here we obtain that $\lim_{t\rightarrow+\infty}x(t)$ exists and this closes the proof.
\end{proof}

Since the class of semi-algebraic functions is closed under addition (see for example \cite{b-sab-teb}) and
$(u,v) \mapsto c\|u-v\|^2$ is semi-algebraic for $c>0$, we can state the above result in case $f+g$ is semi-algebraic.

\begin{remark} The construction of the function $H$, which we used in the above arguments in order to derive a descent property, has been inspired by the
decrease property obtained in \eqref{decr-f}. Similar regularizations of the objective function of \eqref{intr-opt-pb}
have been considered also in \cite{bcl, ipiano}, in the context of the investigation of non-relaxed forward-backward methods involving inertial and
memory effects in the nonconvex setting.
\end{remark}

\begin{remark} (i) We invite the reader to consult \cite[Section 3.2]{bc-forder-kl} for results concerning convergence rates of the trajectory, which are expressed by means of the \L{}ojasiewicz exponent.

(ii) The optimization problem \eqref{intr-opt-pb} can be approached in this non-convex setting also by  a second order dynamical system of implicit-type
\begin{equation}\label{dysy}
\left\{
\begin{array}{ll}
\ddot{x}(t)+\gamma\dot{x}(t)+x(t)=\prox_{\eta f}\big(x(t)-\eta \nabla g(x(t))\big)\\
x(0)=u_0,\,\dot{x}(0)=v_0,
\end{array}
\right.
\end{equation}
where $u_0,v_0\in \R^n$ and $\gamma,\eta\in (0,+\infty)$. A similar analysis can be carried out for this dynamics too, see \cite{bcl-dyn-sec-ord}.

(iii) Finally, let us mention a recent contribution \cite{bk} related to a dynamical system proposed and
investigated in connection with block-structured optimization problems
\begin{equation}\label{opt-pb-structured-nonc}  \ \inf_{(x,y)\in\R^n\times \R^m}[f(x)+g(y)+H(x,y)], \end{equation}
where $f,g$ are proper, lower-semicontinuous and $H$ satisfies certain smoothness conditions. The dynamics
considered in \cite{bk} can be seen as a continuous counter-part of the PALM method introduced in
\cite{b-sab-teb}.
\end{remark}

{\bf Acknowledgments.} The author is thankful to the handling editor and two reviewers for their comments which improved the presentation of the manuscript.

\end{document}